\documentclass[a4paper,11pt]{article}
\usepackage[latin1]{inputenc}
\usepackage[english]{babel}
\usepackage{amsmath}
\usepackage{amsfonts}
\usepackage{amssymb}
\usepackage{epsfig}
\usepackage{amsopn}
\usepackage{amsthm}
\usepackage{color}
\usepackage{graphicx}
\usepackage{subfigure}
\usepackage{enumerate}
\newtheorem{theorem}{Theorem}[section]

\newtheorem{lemma}[theorem]{Lemma}

\newtheorem*{theorem*}{Theorem}
\newtheorem*{lemma*}{Lemma}
\newtheorem*{remark*}{Remark}
\newtheorem*{definition*}{Definition}
\newtheorem*{proposition*}{Proposition}
\newtheorem*{corollary*}{Corollary}
\numberwithin{equation}{section}
\newcommand{\real}{\mathbb{R}}

\let\ced=\c         

\def\qed{\,\unskip\kern 6pt \penalty 500
\raise -2pt\hbox{\vrule \vbox to8pt{\hrule width 6pt
\vfill\hrule}\vrule}\par}
\definecolor{darkblue}{rgb}{0.05, .05, .65}
\definecolor{darkgreen}{rgb}{0.1, .65, .1}
\definecolor{darkred}{rgb}{0.8,0,0}
\newcommand{\beqn}{\begin{equation}}
\newcommand{\eeqn}{\end{equation}}
\newcommand{\bear}{\begin{eqnarray}}
\newcommand{\eear}{\end{eqnarray}}
\newcommand{\bean}{\begin{eqnarray*}}
\newcommand{\eean}{\end{eqnarray*}}
\begin{document}

\title{\huge \bf Global solutions versus finite time blow-up for the supercritical fast diffusion equation with inhomogeneous source}

\author{
\Large Razvan Gabriel Iagar\,\footnote{Departamento de Matem\'{a}tica
Aplicada, Ciencia e Ingenieria de los Materiales y Tecnologia
Electr\'onica, Universidad Rey Juan Carlos, M\'{o}stoles,
28933, Madrid, Spain, \textit{e-mail:} razvan.iagar@urjc.es},\\
[4pt] \Large Ariel S\'{a}nchez,\footnote{Departamento de Matem\'{a}tica
Aplicada, Ciencia e Ingenieria de los Materiales y Tecnologia
Electr\'onica, Universidad Rey Juan Carlos, M\'{o}stoles,
28933, Madrid, Spain, \textit{e-mail:} ariel.sanchez@urjc.es}\\
[4pt] }
\date{}
\maketitle

\begin{abstract}
Solutions in self-similar form, either global in time or presenting finite time blow-up, to the supercritical fast diffusion equation with spatially inhomogeneous source
$$
\partial_tu=\Delta u^m+|x|^{\sigma}u^p, \quad (x,t)\in\real^N\times(0,\infty)
$$
with
$$
m_c=\frac{(N-2)_+}{N}\leq m<1, \quad \sigma\in(\max\{-2,-N\},\infty), \quad p>\max\left\{1+\frac{\sigma(1-m)}{2},1\right\}
$$
are considered. It is proved that global self-similar solutions with the specific tail behavior
$$
u(x,t)\sim C(m)|x|^{-2/(1-m)}, \qquad {\rm as} \ |x|\to\infty
$$
exist exactly for $p\in(p_F(\sigma),p_s(\sigma))$, where
$$
p_F(\sigma)=m+\frac{\sigma+2}{N}, \qquad p_s(\sigma)=\left\{\begin{array}{ll}\frac{m(N+2\sigma+2)}{N-2}, & N\geq3,\\\infty, & N\in\{1,2\}, \end{array}\right.
$$
are the renowned Fujita and Sobolev critical exponents. In contrast, it is shown that self-similar solutions presenting finite time blow-up exist for any $\sigma\in(-2,0)$ and $p$ as above, but do not exist for any $\sigma\geq0$ and $p\in(p_F(\sigma),p_s(\sigma))$. We stress that all these results are \emph{new also in the homogeneous case $\sigma=0$}.
\end{abstract}

\


\smallskip


\section{Introduction}

The fast diffusion equation
\begin{equation}\label{FDE}
u_t=\Delta u^m, \qquad 0<m<1,
\end{equation}
has established itself in the last half century as one of the central models in the theory of nonlinear diffusion equations. A good monograph on this subject is the book \cite{VazSmooth}, and we also refer the reader to the recent survey by Bonforte and Figalli \cite{BF24}, where also some applications in physics are enumerated. In particular, the mathematical analysis of the fast diffusion equation depends on a fundamental exponent known as the \emph{critical exponent}
\begin{equation*}
m_c=\frac{(N-2)_{+}}{N},
\end{equation*}
splitting the interval $(0,1)$ into the \emph{supercritical range} $m_c<m<1$ and the \emph{subcritical range} $0<m<m_c$. In the supercritical range $m>m_c$, as well as for $m=m_c$, Eq. \eqref{FDE} maintains some of the features of the heat equation, conserving the total mass (or $L^1$ norm) of the initial condition along the evolution, while the subcritical range produces typically the phenomenon of \emph{finite time extinction} of the solutions, due to a loss of mass at infinity, a rather unexpected phenomenon well explained in \cite[Section 5.5]{VazSmooth}. Finer qualitative differences between the two ranges have been established in papers such as \cite{BV10, BIV10}.

A significant feature of nonlinear diffusion equations is the availability of solutions in radially symmetric self-similar form, either in forward form, that is
\begin{equation}\label{forward.SS}
u(x,t)=t^{-\alpha}f(\xi), \qquad \xi=|x|t^{-\beta},
\end{equation}
or in backward form, that is
\begin{equation}\label{backward.SS}
u(x,t)=(T-t)^{-\alpha}f(\xi), \qquad \xi=|x|(T-t)^{-\beta}, \qquad T\in(0,\infty),
\end{equation}
for suitable exponents $\alpha$ and $\beta$. Solutions as in \eqref{forward.SS} are global in time and usually prototypes of behavior of general solutions as $t\to\infty$, while solutions as in \eqref{backward.SS} model finite time blow-up or finite time extinction. Just to quote a few number of works from the rather huge literature devoted to the role of self-similar solutions in developing the theory of nonlinear diffusion equations, we refer the reader to papers such as \cite{Le96, GV97, FT00, BBDGV, IL14, BIS16} and references therein. Let us stress here that, in all these works (and others), it has been noticed that the most interesting self-similar solutions in the study of a fast diffusion equation are the ones presenting \emph{the fastest possible decay} as $|x|\to\infty$. As an example, a celebrated branch of self-similar solutions to Eq. \eqref{FDE} are the \emph{anomalous solutions} obtained rigorously in \cite{PZ95} following the more formal analysis in \cite{Ki93}.

The present work is devoted to the classification of possible behaviors (at the level of self-similar solutions) of solutions to the supercritical fast diffusion equation with a source
\begin{equation}\label{eq1}
u_t=\Delta u^m+|x|^{\sigma}u^p, \quad (x,t)\in\real^N\times(0,\infty),
\end{equation}
with $m\in[m_c,1)$, $N\geq1$ and $\sigma\in(\max\{-2,-N\},\infty)$. This equation involves a competition between the fast diffusion equation and a source term, raising the question whether the increase in mass introduced by the source term is sufficiently strong to produce finite time blow-up. In this work, we restrict ourselves to the range
\begin{equation*}
p>\max\{1,p_L(\sigma)\}, \quad p_L(\sigma)=1+\frac{\sigma(1-m)}{2},
\end{equation*}
since the complementary range has been considered in our recent works \cite{IS22c, IMS23b}.

Eq. \eqref{eq1} has been thoroughly studied in the semilinear case $m=1$ and the slow diffusion case $m>1$, especially with a homogeneous source term, that is, with $\sigma=0$. The main emphasis in this long-term study has been the phenomenon of finite time blow-up and its finer properties (blow-up sets, rates, profiles), see for example the monographs \cite{QS} for $m=1$ and \cite{S4} for $m>1$. The seminal work by Fujita \cite{Fu66} and further extensions of it led to the critical exponent known nowadays as \emph{the Fujita exponent}, which in the case of Eq. \eqref{eq1} writes
\begin{equation*}
p_F(\sigma)=m+\frac{2+\sigma}{N}
\end{equation*}
limiting the range where all non-trivial solutions blow up in finite time, that is, $1<p\leq p_F(\sigma)$, and the range where global solutions exist, $p>p_F(\sigma)$. The semilinear case of Eq. \eqref{eq1} with $\sigma>0$ had been considered in a number of older papers such as \cite{BK87, Pi97, Pi98} devoted to general qualitative properties and the life span of the solutions. Filippas and Tertikas \cite{FT00} performed an analysis of self-similar solutions to Eq. \eqref{eq1} with $m=1$. Mukai and Seki \cite{MS21} described, also for $m=1$ and $\sigma\in(-2,\infty)$, the phenomenon of blow-up of Type II, that is, with variable rates and geometric patterns, which is characteristic to larger values of $p>1$. The authors and their collaborators started in recent years a larger project of understanding the influence of the weight $|x|^{\sigma}$ in Eq. \eqref{eq1}, and in the range $m>1$ a number of results concerning the analytical properties of self-similar solutions have been obtained, see for example \cite{IS22b, IMS23, ILS23, IL22} and references therein. It has been noticed in these works that the sign of the expression
\begin{equation}\label{const.L}
L:=\sigma(m-1)+2(p-1)
\end{equation}
strongly influences the dynamics of Eq. \eqref{eq1}. In particular, our limitation $p>p_L(\sigma)$ is equivalent to $L>0$, a range which includes the case of homogeneous source term $\sigma=0$.

The fast diffusion range of Eq. \eqref{eq1} has been less considered in papers, most of these works being centered in identifying whether the Fujita exponent $p_F(\sigma)$ keeps playing its role of separating the two regimes as explained above, see for example \cite{Qi93, MM95, GuoGuo01} for $\sigma=0$ and \cite{Qi98} for $\sigma\in(-2,\infty)$. The latter work \cite{Qi98} gives also an example of a global self-similar solution in forward form \eqref{forward.SS} for $p>p_F(\sigma)$. Maing\'e \cite{Ma08} extended the results of \cite{GuoGuo01} in order to describe the connection between the decay rate of an initial condition $u_0$ as $|x|\to\infty$ and the time frame of existence of the solution to the Cauchy problem with data $u_0$ for the whole fast diffusion range $m\in(0,1)$. More recently, the fast diffusion equation with localized weight and in dimension $N=1$ has been considered in \cite{BZZ11}.

Justified by this lack of previous results, we started working on the fast diffusion range of Eq. \eqref{eq1} and found some rather unexpected results in our previous works \cite{IS22c, IMS23b}. In the former, the critical case $p=p_L(\sigma)$, that is, $L=0$, is addressed, and \emph{eternal} self-similar solutions in exponential form are found and classified, while the latter work deals with the case $1<p<p_L(\sigma)$ (which automatically implies $\sigma>0$). In both works, existence of solutions in self-similar form is limited to the subcritical range $m\in(0,m_c)$ and falls out of the present analysis. However, the behavior of the solutions as $|x|\to\infty$ depends on two important critical exponents defined as
\begin{equation*}
p_c(\sigma):=\left\{\begin{array}{ll}\frac{m(N+\sigma)}{N-2}, & N\geq3,\\ \infty, & N\in\{1,2\},\end{array}\right.
\qquad p_s(\sigma):=\left\{\begin{array}{ll}\frac{m(N+2\sigma+2)}{N-2}, & N\geq3,\\ \infty, & N\in\{1,2\},\end{array}\right.
\end{equation*}
which will also come decisively into play in the current work. We are now ready to introduce and explain our results.

\medskip

\noindent \textbf{Main results.} As previously explained, our goal is to identify the optimal ranges of existence and non-existence and then classify (according to their behavior as $|x|\to\infty$) the self-similar solutions to Eq. \eqref{eq1}, for $m\in[m_c,1)$, $\sigma\in(\max\{-2,-N\},\infty)$ and $p>\max\{1,p_L(\sigma)\}$, in one of the forms \eqref{forward.SS} or \eqref{backward.SS}. By replacing these ansatz into Eq. \eqref{eq1} and after straightforward calculations, we find that in both cases the self-similar exponents are
\begin{equation}\label{SS.exp}
\alpha=\frac{\sigma+2}{L}>0, \qquad \beta=\frac{p-m}{L}>0,
\end{equation}
where $L>0$ is the constant defined in \eqref{const.L}. 

\medskip

\noindent \textbf{A. Global self-similar solutions}. We are looking for self-similar solutions in forward form \eqref{forward.SS}, with $\alpha$, $\beta$ as in \eqref{SS.exp}. A direct calculation gives that the profiles $f$ solve the differential equation
\begin{equation}\label{ODE.forward}
(f^m)''(\xi)+\frac{N-1}{\xi}(f^m)'(\xi)+\alpha f(\xi)+\beta\xi f'(\xi)+\xi^{\sigma}f^p(\xi)=0.
\end{equation}
We will show that the local behavior at $\xi=0$ is given by the following local expansion:
\begin{equation}\label{beh.P0f}
f(\xi)\sim\left\{\begin{array}{ll}\left[D+\frac{\alpha(1-m)}{2mN}\xi^2\right]^{-1/(1-m)}, & \sigma>0,\\[1mm]
\left[D+\frac{(1-m)\alpha(1+\alpha D^{(p-1)/(m-1)})}{2mN}\xi^2\right]^{-1/(1-m)}, & \sigma=0,\\[1mm]
\left[D+\frac{p-m}{m(N+\sigma)(\sigma+2)}\xi^{\sigma+2}\right]^{-1/(p-m)}, & \sigma\in(-2,0),
\end{array}\right. \ \ {\rm as} \ \xi\to0, \ \ D>0.
\end{equation}
With this preparation, we are ready to give the following sharp classification result.
\begin{theorem}\label{th.global.super}
Let $m\in[m_c,1)$ and $\sigma\in(\max\{-2,-N\},\infty)$. Then
\begin{enumerate}
  \item If $m>m_c$, then for any $p\in(\max\{1,p_F(\sigma)\},\max\{1,p_s(\sigma)\})$, there exists at least a global solution in the form \eqref{forward.SS} whose profile satisfies the local behavior \eqref{beh.P0f} as $\xi\to0$ and presents the following fast decay
\begin{equation}\label{beh.P3}
f(\xi)\sim\left[\frac{2m(mN-N+2)}{1-m}\right]^{1/(1-m)}\xi^{-2/(1-m)}, \qquad {\rm as} \ \xi\to\infty,
\end{equation}
and infinitely many global solutions in the form \eqref{forward.SS} whose profiles present the same local behavior as $\xi\to0$ and the slower decay at infinity
\begin{equation}\label{beh.Q1f}
f(\xi)\sim K\xi^{-(\sigma+2)/(p-m)}, \qquad {\rm as} \ \xi\to\infty, \qquad K>0.
\end{equation}
  \item If $m=m_c$, then for any $p\in(\max\{1,p_F(\sigma)\},\max\{1,p_s(\sigma)\})$, all the above hold true, but the fast decay \eqref{beh.P3} as $\xi\to\infty$ is replaced by the following one involving a logarithmic correction
\begin{equation}\label{beh.P3crit}
f(\xi)\sim C_0\xi^{-N}(\ln\,\xi)^{-N/2}, \qquad C_0=\left[\frac{N}{(N-2)^2}\right]^{-N/2}.
\end{equation}
  \item If $N\geq3$, for any $p\geq\max\{1,p_s(\sigma)\}$ there are infinitely many global solutions in forward form \eqref{forward.SS} whose profiles present the local behavior \eqref{beh.P0f} as $\xi\to0$ and the slow decay \eqref{beh.Q1f} as $\xi\to\infty$, but no solutions with decay as in \eqref{beh.P3}.
  \item If $p\leq p_F(\sigma)$, there are no global solutions in forward form \eqref{forward.SS} to Eq. \eqref{eq1} at all.
\end{enumerate}
\end{theorem}
Let us remark first that the behavior \eqref{beh.P3} is \emph{specific to the supercritical fast diffusion}. Indeed, it cannot exist either for $m\geq 1$ or in the subcritical range of the fast diffusion (as we shall see in a future work), and in the range $m\in[m_c,1)$ it is a faster decay than \eqref{beh.Q1f}, since
$$
\frac{2}{1-m}-\frac{\sigma+2}{p-m}=\frac{L}{(p-m)(1-m)}>0.
$$
This behavior seems to have been left unnoticed for the fast diffusion with source terms, although it is similar to the one of the ZKB solutions of the pure fast diffusion \eqref{FDE} given in \cite[Section 2.1]{VazSmooth}, see also \cite{BBDGV} for an analysis of their role in the large time behavior of solutions to \eqref{FDE}. The same behavior \eqref{beh.P3} has been identified, for $\sigma=0$, as the decay of very singular solutions to the fast diffusion equation with absorption in the supercritical range \cite{Le96}. The last item in Theorem \ref{th.global.super} is expected, in the sense that it has been shown in previous works such as \cite{Qi98, Su02} that there are in general no global solutions for $1<p\leq p_F(\sigma)$, and we put it in the statement for the sake of completeness. However, the proof of Theorem \ref{th.global.super} allows us to get a new understanding of the Fujita exponent as a separatrix between two different regimes of behavior. We also observe that the classification given in Theorem \ref{th.global.super} is sharp, as all the values of $p$ are covered and explicit critical exponents limit the ranges of existence of the possible behaviors.

\medskip

\noindent \textbf{B. Self-similar solutions with blow-up}. We are looking for self-similar solutions to Eq. \eqref{eq1} in backward form \eqref{backward.SS}, with $\alpha$ and $\beta$ as in \eqref{SS.exp} and whose profiles solve the differential equation
\begin{equation}\label{ODE.backward}
(f^m)''(\xi)+\frac{N-1}{\xi}(f^m)'(\xi)-\alpha f(\xi)-\beta\xi f'(\xi)+\xi^{\sigma}f^p(\xi)=0.
\end{equation}
The next theorem extends to the supercritical fast diffusion the analysis done in \cite{FT00} for $m=1$, although the techniques of the proofs are very different.
\begin{theorem}\label{th.blowup.super}
Let $m\in[m_c,1)$. Then we have:
\begin{enumerate}
  \item For any $\sigma\in(\max\{-2,-N\},0)$ and $p\in(1,\max\{1,p_s(\sigma)\})$, there exists at least one profile $f(\xi)$ solution to Eq. \eqref{ODE.backward} for self-similar solutions in backward form \eqref{backward.SS} with $\alpha$ and $\beta$ as in \eqref{SS.exp} and presenting finite time blow-up. It has the local behavior \eqref{beh.P0f} for $\sigma<0$ and the decay \eqref{beh.Q1f} as $\xi\to\infty$.
  \item In the homogeneous case $\sigma=0$ and $p\in(1,p_s(0))$, there is no self-similar solution in backward form \eqref{backward.SS} to Eq. \eqref{eq1}.
  \item For any $\sigma>0$ and $p\in(p_F(\sigma),p_s(\sigma))$, there is no self-similar solution in backward form \eqref{backward.SS} to Eq. \eqref{eq1}. The same non-existence result holds true for $p\in(p_L(\sigma),p_*(\sigma))$ provided $N\geq4$, where
      $$
      p_*(\sigma)=1+\frac{m\sigma}{N-2}.
      $$
\end{enumerate}
\end{theorem}

\noindent \textbf{Remark.} We strongly believe that, in fact, non-existence of blow-up self-similar solutions holds true for any $\sigma>0$ and $p\in(p_L(\sigma),p_s(\sigma))$, similar to the case $\sigma=0$ (where we can prove it completely), and that it is only a technical problem that does not allow us to complete the proof also in the remaining range $p_*(\sigma)\leq p\leq p_F(\sigma)$. Actually, if $N\geq4$ and $\sigma>N-2$, we readily observe that $p_*(\sigma)>p_F(\sigma)$, hence for this range, our statement covers all the interval of $p$.

We intentionally left aside in the present work the range $p\geq p_s(\sigma)$, where new critical exponents have been identified (such as the Joseph-Lundgren and Lepin exponents) and whose analysis is much more involved, see for example results in \cite{GV97, IS25} for $\sigma=0$ and $m>1$. Coming back to the results in \cite{FT00} for $m=1$, their proof of the non-existence item is based on a Pohozaev identity which exploits the linearity of the dominant order in the equation. This approach does not work with $m<1$, thus we prove the non-existence through a geometric approach in a phase space of a dynamical system, by constructing sharp separatrices limiting its trajectories.

\medskip

\noindent \textbf{A family of stationary solutions for} $\mathbf{p=p_s(\sigma)}$. If $N\geq3$, $\sigma>-2$ and $p=p_s(\sigma)$ there exists a one-parameter family of explicit \emph{stationary} solutions
\begin{equation*}
U_C(x)=\left[\frac{(N-2)(N+\sigma)C}{(|x|^{\sigma+2}+C)^2}\right]^{(N-2)/2m(\sigma+2)}, \quad C>0.
\end{equation*}
Notice that $U_C(0)\in(0,\infty)$ and $U_C'(0)=0$. These solutions have been obtained already in our previous work \cite[Section 9]{IMS23b}, following a particular case previously discovered in \cite[Section 3.5]{IS22c} for $p=p_L(\sigma)$ and $m=(N-2)_+/(N+2)$, but they are still valid as solutions in our range of parameters.

\medskip

\noindent \textbf{Remarks. 1}. The order of the critical exponents is implicitly understood in the statements of the theorems. Indeed, since $m\geq m_c$, we have
\begin{equation*}
\begin{split}
&p_F(\sigma)-p_L(\sigma)=\frac{(\sigma+2)(m-m_c)}{2}>0, \ \ p_c(\sigma)-p_F(\sigma)=\frac{(\sigma+2)(m-m_c)}{N-2}>0, \\ &p_s(\sigma)-p_c(\sigma)=\frac{m(\sigma+2)}{N-2}>0,
\end{split}
\end{equation*}
hence $p_L(\sigma)<p_F(\sigma)<p_c(\sigma)<p_s(\sigma)$. Moreover, observe that $p_s(\sigma)>1$ if and only if 
$$
m>\frac{N-2}{N+2\sigma+2},
$$
and the right hand side tends to one as $\sigma\to-2$. If $\sigma$ is such that $p_s(\sigma)<1$, then the statement for $p>p_s(\sigma)$ is in force for any $p>1$ in Theorem \ref{th.global.super}.  

\textbf{2}. All our results appear to be new even for $\sigma=0$, which is just a particular case in the analysis.

\medskip

\noindent \textbf{A comment on the regularity of the self-similar solutions}. Since all the self-similar profiles identified in Theorems \ref{th.global.super} and \ref{th.blowup.super} are strictly positive for any $\xi>0$, we readily infer that the corresponding solutions are of class $C^2(\real^N\setminus\{0\})$ and the only problem might occur at $\xi=0$. A quick inspection of the local behavior up to second order gives then that the solutions with profiles as in \eqref{beh.P0f} as $\xi\to0$:

$\bullet$ are of class $C^2$ also at $x=0$ when $\sigma\geq0$ and thus are classical solutions.

$\bullet$ belong to $C^1(\real^N)$ but not to $C^2(\real^N)$ for $\sigma\in(-1,0)$. One can check that the subsequent self-similar solutions are then weak solutions to Eq. \eqref{eq1}, but not classical ones.

$\bullet$ are no longer of class $C^1$ at $\xi=0$ and form a peak for $\sigma\in(-2,-1)$. For the subsequent self-similar solutions $u$ we notice that the singularity of $\nabla u^m$ at $x=0$ is of order $|x|^{\sigma+1}$ and we infer that $\nabla u^m\in L^2(\real^N)$ and thus they are still weak solutions to Eq. \eqref{eq1}. We expect them as well to play an important role in the dynamics of the equation. Such solutions with peaks at the origin have been also found in equations involving a singular weight, see \cite{RV06}, being therein essential for the large time behavior. A more precise description of the optimal regularity at $x=0$ of these self-similar solutions seems to be similar to the one in \cite[Section 3.3]{IL25}.

\section{The supercritical range. Phase plane and local analysis}\label{sec.local}

We begin our study with the quest for global solutions in the form \eqref{forward.SS}, whose profiles solve the differential equation \eqref{ODE.forward}. To ease the analysis, we work in dimension $N\geq3$ and postpone dimensions $N=1$ and $N=2$ to Section \ref{sec.N12}. We fix throughout this section $m\in(m_c,1)$, the critical case $m=m_c$ being addressed at the end of it. In order to classify its solutions, we perform the following change of variables, already used in \cite{IMS23b} for the subcritical fast diffusion range:
\begin{equation}\label{PSchange}
X(\eta)=\frac{\alpha}{m}\xi^2f^{1-m}(\xi), \qquad Y(\eta)=\frac{\xi f'(\xi)}{f(\xi)}, \qquad Z(\eta)=\frac{1}{m}\xi^{\sigma+2}f^{p-m}(\xi),
\end{equation}
where $\eta=\ln\,\xi$ is the new independent variable and $\alpha$ is given in \eqref{SS.exp}. Eq. \eqref{ODE.forward} is transformed, after direct calculations, into the quadratic dynamical system
\begin{equation}\label{PSsyst}
\left\{\begin{array}{ll}\dot{X}=X(2+(1-m)Y), \\ \dot{Y}=-X-(N-2)Y-Z-mY^2-\frac{p-m}{\sigma+2}XY, \\ \dot{Z}=Z(\sigma+2+(p-m)Y).\end{array}\right.
\end{equation}
We thus analyze the orbits in the phase space associated to the system \eqref{PSsyst} and extract from them information about the profiles and their tail behavior. To this end, one needs to study the critical points of the system, both finite and infinite. At some points we will be rather brief, as many technical details in the forthcoming analysis are identical to the ones already given in great detail in \cite{IMS23b}. Notice first that in our range of interest (non-negative profiles) we have $X\geq0$, $Z\geq0$ and the planes $\{X=0\}$ and $\{Z=0\}$ are invariant for the system \eqref{PSsyst}.

\bigskip

\noindent \textbf{Finite critical points of the system \eqref{PSsyst}}. A simple inspection of the system \eqref{PSsyst} shows that there are four finite critical points in the range $m\in(m_c,1)$, namely
\begin{equation*}
\begin{split}
P_0=(0,0,0), \ \ &P_1=\left(0,-\frac{N-2}{m},0\right), \ \ P_2=\left(0,-\frac{\sigma+2}{p-m},\frac{(N-2)(\sigma+2)(p-p_c(\sigma))}{(p-m)^2}\right),\\
&P_3=\left(\frac{2(\sigma+2)(mN-N+2)}{L(1-m)},-\frac{2}{1-m},0\right).
\end{split}
\end{equation*}
We give below the local analysis of these critical points when $m\in(m_c,1)$, mentioning that the point $P_2$ only exists for $p>p_c(\sigma)$.
\begin{lemma}\label{lem.P0}
In a neighborhood of the critical point $P_0$, the system \eqref{PSsyst} has a two-dimensional unstable manifold and a one-dimensional stable manifold. The profiles contained in orbits going out of $P_0$ satisfy as $\xi\to0$ either the local behavior \eqref{beh.P0f}.
\end{lemma}
\begin{proof}
The linearization of the system \eqref{PSsyst} in a neighborhood of $P_0$ has the matrix
$$
M(P_0)=\left(
         \begin{array}{ccc}
           2 & 0 & 0 \\
           -1 & -(N-2) & -1 \\
           0 & 0 & \sigma+2 \\
         \end{array}
       \right),
$$
with two positive eigenvalues $\lambda_1=2$ and $\lambda_3=\sigma+2$ and one negative eigenvalue $\lambda_2=-(N-2)$ under the assumption $N\geq3$. Since the second eigenvector is $e_2=(0,1,0)$, the stable manifold is contained in the $Y$ axis, while for the unstable manifold, in a first order approximation, we readily get from the first and third equation that
\begin{equation}\label{interm1}
Z(\eta)\sim CX^{(\sigma+2)/2}(\eta), \qquad {\rm as} \ \eta\to-\infty.
\end{equation}
Substituting \eqref{interm1} into the second equation and neglecting the quadratic terms, we get after integration that
\begin{equation}\label{interm2}
Y(\eta)\sim-\frac{X(\eta)}{N}-\frac{Z(\eta)}{N+\sigma}+C_1X(\eta)^{-(N-2)/2}, \qquad {\rm as} \ \eta\to-\infty.
\end{equation}
Since the trajectories are assumed to pass by $P_0$, we have to let $C_1=0$ in \eqref{interm2}. Notice then that \eqref{interm1} implies that
\begin{equation}\label{interm3}
\lim\limits_{\eta\to-\infty}\frac{Z(\eta)}{X(\eta)}=\left\{\begin{array}{lll}0, & {\rm if} \ \sigma>0,\\ C\in(0,\infty), & {\rm if} \ \sigma=0, \\ +\infty, & {\rm if} \ \sigma<0.\end{array}\right.
\end{equation}
We then go back to the profiles by keeping the dominating term among $X(\eta)$ and $Z(\eta)$ in \eqref{interm2}, according to \eqref{interm3}, and then undoing the change of variable \eqref{PSchange} and noticing that $\eta\to-\infty$ translates into $\xi\to0$. An integration over $(0,\xi)$ for $\xi>0$ leads to the local behavior \eqref{beh.P0f}.
\end{proof}
The local analysis of the points $P_1$ and $P_2$ follows below.
\begin{lemma}\label{lem.P1}
The critical point $P_1$ is an unstable node if $p\in(p_L(\sigma),p_c(\sigma))$ and has a two-dimensional unstable manifold contained in the invariant plane $\{Z=0\}$ and a one-dimensional stable manifold contained in the invariant plane $\{X=0\}$ if $p>p_c(\sigma)$. The orbits going out of it for $p\in(p_L(\sigma),p_c(\sigma))$ contain profiles presenting a vertical asymptote
\begin{equation}\label{beh.P1.super}
f(\xi)\sim C\xi^{-(N-2)/m}, \qquad {\rm as} \ \xi\to0, \qquad C>0.
\end{equation}
\end{lemma}
\begin{proof}
The linearization of the system \eqref{PSsyst} in a neighborhood of $P_1$ has the matrix
$$
M(P_1)=\left(
         \begin{array}{ccc}
           \frac{mN-N+2}{m} & 0 & 0 \\[1mm]
           \frac{(N-2)(p-p_c(\sigma))}{m(\sigma+2)} & N-2 & -1 \\[1mm]
           0 & 0 & \frac{(N-2)(p_c(\sigma)-p)}{m} \\
         \end{array}
       \right),
$$
with eigenvalues
$$
\lambda_1=\frac{mN-N+2}{m}>0, \qquad \lambda_2=N-2>0, \qquad \lambda_3=\frac{(N-2)(p_c(\sigma)-p)}{m}.
$$
Notice that, if $p\in(p_L(\sigma),p_c(\sigma))$ then the three eigenvalues are positive and we have an unstable node. We then readily deduce the local behavior \eqref{beh.P1.super} by noticing that $\lim\limits_{\eta\to-\infty}Y(\eta)=-(N-2)/m$ and undoing the change of variable \eqref{PSchange}. On the contrary, if $p>p_c(\sigma)$, then $\lambda_3<0$ and a simple calculation of the eigenvectors of the matrix $M(P_1)$, together with the invariance of the planes $\{X=0\}$ and $\{Z=0\}$, lead to the conclusion in this case.
\end{proof}
\begin{lemma}\label{lem.P2}
Let $p>p_c(\sigma)$. Then the critical point $P_2$ is
\begin{itemize}
\item an unstable node or focus if $p\in(p_c(\sigma),p_s(\sigma))$.
\item a saddle point with a stable two-dimensional manifold contained in the invariant plane $\{X=0\}$ and a unique orbit going out of it if $p>p_s(\sigma)$.
\end{itemize}
In any of these two cases, the profiles contained in the unstable manifold of the point $P_2$ present a vertical asymptote in the form
\begin{equation}\label{beh.P2.super}
f(\xi)\sim K\xi^{-(\sigma+2)/(p-m)}, \ {\rm as} \ \xi\to0, \qquad K=\left[\frac{m(\sigma+2)(N-2)(p-p_c(\sigma))}{(p-m)^2}\right]^{1/(p-m)}.
\end{equation}
\end{lemma}
\begin{proof}
The linearization of the system \eqref{PSsyst} in a neighborhood of the point $P_2$ has the matrix
$$
M(P_2)=\left(
         \begin{array}{ccc}
           \frac{L}{p-m} & 0 & 0 \\[1mm]
           0 & -\frac{(N-2)(p-p_s(\sigma))}{p-m} & -1 \\[1mm]
           0 & \frac{(N-2)(\sigma+2)(p-p_c(\sigma))}{p-m} & 0 \\
         \end{array}
       \right),
$$
where $L$ is defined in \eqref{const.L}, and whose eigenvalues satisfy
$$
\lambda_1=\frac{L}{p-m}>0, \ \lambda_2+\lambda_3=-\frac{(N-2)(p-p_s(\sigma))}{p-m}, \ \lambda_2\lambda_3=\frac{(N-2)(\sigma+2)(p-p_c(\sigma))}{p-m}>0.
$$
Notice that $\lambda_2+\lambda_3$ changes sign exactly at $p=p_s(\sigma)$, leading readily to the classification in the statement. Moreover, the orbits contained in the unstable manifold (either three- or one-dimensional) of this point have
\begin{equation}\label{zp2}
Z(\eta)\to Z(P_2):=\frac{(N-2)(\sigma+2)(p-p_c(\sigma))}{(p-m)^2}, \qquad {\rm as} \ \eta\to-\infty,
\end{equation}
whence we immediately get the local behavior \eqref{beh.P2.super} as $\xi\to0$.
\end{proof}

\noindent \textbf{Remark.} If $p>p_s(\sigma)$, the only orbit contained in the unstable manifold of $P_2$ is explicit, and it contains the stationary solution
\begin{equation*}
u(x)=K|x|^{-(\sigma+2)/(p-m)}, \qquad K=\left[\frac{m(\sigma+2)(N-2)(p-p_c(\sigma))}{(p-m)^2}\right]^{1/(p-m)}.
\end{equation*}

The critical point $P_3$ encloses a behavior which is specific to this range.
\begin{lemma}\label{lem.P3}
For any $m\in(m_c,1)$, the critical point $P_3$ is a saddle point with a two-dimensional stable manifold and a one-dimensional unstable manifold, the latter being fully contained in the invariant plane $\{Z=0\}$. The profiles contained in orbits belonging to the two-dimensional stable manifold present the local behavior \eqref{beh.P3} as $\xi\to\infty$.
\end{lemma}
\begin{proof}
The linearization of the system \eqref{PSsyst} in a neighborhood of $P_3$ has the matrix
$$\left(
  \begin{array}{ccc}
    0 & \frac{2(\sigma+2)(mN-N+2)}{L} & 0 \\[1mm]
    \frac{L}{(1-m)(\sigma+2)} & A(m,N,p,\sigma) & -1 \\[1mm]
    0 & 0 & -\frac{L}{1-m} \\
  \end{array}
\right),$$
where $L>0$ is defined in \eqref{const.L} and
$$
A(m,N,p,\sigma)=\frac{(1-m)^2N(\sigma+2)+2(m^2-1)\sigma+4(mp-1)}{L(1-m)}.
$$
We observe that the eigenvalues of $M(P_3)$ satisfy
$$
\lambda_3=-\frac{L}{1-m}<0, \qquad \lambda_1+\lambda_2=A(m,N,p,\sigma), \qquad \lambda_1\lambda_2=-\frac{2(mN-N+2)}{1-m}<0,
$$
hence there are in any case two negative eigenvalues and a positive one. Moreover, it is obvious that the eigenvectors corresponding to both $\lambda_1$ and $\lambda_2$ (hence, of the unique positive eigenvalue) are contained in the invariant plane $\{Z=0\}$, which gives that the whole (unique) unstable orbit stays in this plane. Finally, the two-dimensional stable manifold contains profiles such that
$$
X(\eta)\to X(P_3):=\frac{2(\sigma+2)(mN-N+2)}{L(1-m)}, \qquad {\rm as} \ \eta\to\infty,
$$
which easily gives \eqref{beh.P3} as $\xi\to\infty$ by undoing the change of variable \eqref{PSchange}.
\end{proof}

\medskip

\noindent \textbf{Critical points at infinity for the system \eqref{PSsyst}.} The critical points at infinity of the system \eqref{PSsyst} are analyzed via the Poincar\'e hypersphere. To this end, we follow the theory in \cite[Section 3.10]{Pe} and let
$$
X=\frac{\overline{X}}{W}, \qquad Y=\frac{\overline{Y}}{W}, \qquad Z=\frac{\overline{Z}}{W}.
$$
The critical points at space infinity, expressed in the new variables $(\overline{X},\overline{Y},\overline{Z},W)$, are then given as in \cite[Theorem 4, Section 3.10]{Pe} by the following system
\begin{equation*}
\left\{\begin{array}{ll}\frac{1}{\sigma+2}\overline{X}\overline{Y}[(p-m)\overline{X}-(\sigma+2)\overline{Y}]=0,\\
(p-1)\overline{X}\overline{Z}\overline{Y}=0,\\
\frac{1}{\sigma+2}\overline{Y}\overline{Z}[p(\sigma+2)\overline{Y}-(p-m)\overline{X}]=0,\end{array}\right.
\end{equation*}
together with the condition of belonging to the equator of the hypersphere, which leads to $W=0$ and the additional equation $\overline{X}^2+\overline{Y}^2+\overline{Z}^2=1$. Since $\overline{X}\geq0$ and $\overline{Z}\geq0$, we find the following critical points at infinity:
\begin{equation*}
\begin{split}
&Q_1=(1,0,0,0), \ \ Q_{2,3}=(0,\pm1,0,0), \ \ Q_4=(0,0,1,0), \ \ Q_{\gamma}=\left(\gamma,0,\sqrt{1-\gamma^2},0\right),\\
&Q_5=\left(\frac{\sigma+2}{\sqrt{(\sigma+2)^2+(p-m)^2}},\frac{p-m}{\sqrt{(\sigma+2)^2+(p-m)^2}},0,0\right),
\end{split}
\end{equation*}
with $\gamma\in(0,1)$. We analyze these critical points below, being rather sketchy and referring the reader to details and calculations given in \cite{IMS23b} where it applies. For the critical points $Q_1$, $Q_5$ and $Q_{\gamma}$, we employ the system obtained following the projection on the $X$ variable according to \cite[Theorem 5(a), Section 3.10]{Pe}, namely
\begin{equation}\label{PSinf1}
\left\{\begin{array}{ll}\dot{x}=x[(m-1)y-2x],\\
\dot{y}=-y^2-\frac{p-m}{\sigma+2}y-x-Nxy-xz,\\
\dot{z}=z[(p-1)y+\sigma x],\end{array}\right.
\end{equation}
where the variables expressed in lowercase letters are obtained from the original ones by the following change of variable
\begin{equation}\label{change2}
x=\frac{1}{X}, \qquad y=\frac{Y}{X}, \qquad z=\frac{Z}{X},
\end{equation}
and the independent variable with respect to which derivatives are taken in \eqref{PSinf1} is defined implicitly by the differential equation
\begin{equation*}
\frac{d\eta_1}{d\xi}=\frac{\alpha}{m}\xi f^{1-m}(\xi).
\end{equation*}
In this system, the critical points $Q_1$, $Q_5$ and $Q_{\gamma}$ are identified as
$$
Q_1=(0,0,0), \ \ Q_5=\left(0,-\frac{p-m}{\sigma+2},0\right), \ \ Q_{\gamma}=(0,0,\kappa), \ \ {\rm with} \ \kappa=\frac{\sqrt{1-\gamma^2}}{\gamma}\in(0,\infty).
$$
We gather the local analysis in a neighborhood of them in the following technical result.
\begin{lemma}\label{lem.Q1Q5Qg}
\begin{enumerate}
  \item The critical point $Q_1$ has a two-dimensional center manifold on which the flow is stable, and a one-dimensional stable manifold. It behaves as a stable node for orbits arriving from the finite part of the phase space associated to the system \eqref{PSsyst}. The corresponding profiles present the local behavior \eqref{beh.Q1f} as $\xi\to\infty$.
  \item The critical point $Q_5$ is a saddle point with a two-dimensional unstable manifold fully contained in the invariant plane $\{z=0\}$ of the system \eqref{PSinf1}, that is, also in the invariant plane $\{Z=0\}$ of the system \eqref{PSsyst}, and a one-dimensional stable manifold contained in the invariant plane $\{x=0\}$ of the system \eqref{PSinf1}. There are no profiles contained in these orbits.
  \item For any $\gamma\in(0,1)$, the critical point $Q_{\gamma}$ only admits orbits fully contained in the invariant plane $\{x=0\}$ of the system \eqref{PSinf1}.
\end{enumerate}
\end{lemma}
\begin{proof}
\textbf{1.} The linearization of the system \eqref{PSinf1} in a neighborhood of $Q_1$ has the matrix
$$M(Q_1)=\left(
         \begin{array}{ccc}
           0 & 0 & 0 \\[1mm]
           1 & -\frac{p-m}{\sigma+2} & 0 \\[1mm]
           0 & 0 & 0 \\
         \end{array}
       \right).
$$
We analyze the center manifolds of $Q_1$ by further introducing the change of variable
$$
w=\frac{p-m}{\sigma+2}y+x,
$$
in order to put the system \eqref{PSinf1} in the canonical form for an application of the center manifold theory according to \cite{Carr} (see also \cite[Section 2.12]{Pe}), which in our case writes
\begin{equation*}
\left\{\begin{array}{ll}\dot{x}=-\frac{1}{\beta}x^2+\frac{(m-1)\alpha}{\beta}xw,\\[1mm]
\dot{w}=-\frac{\beta}{\alpha}w-\frac{\alpha}{\beta}w^2-\frac{\beta}{\alpha}xz-\frac{(m+1)\alpha-N\beta}{\beta}xw-\frac{m\alpha-(N-2)\beta}{\beta}x^2,\\[1mm]
\dot{z}=-\frac{1}{\beta}xz+\frac{(p-1)\alpha}{\beta}zw.\end{array}\right.
\end{equation*}
Following the calculations given in \cite[Lemma 2.1]{IMS23} (see also \cite[Lemma 3.1]{IMS23b}), we get the equation of the center manifold of $Q_1$ as
\begin{equation*}
w=-\frac{(\sigma+2)(N-2)(p_c(\sigma)-p)}{(p-m)^2}x^2-xz+o(|(x,z)|^2),
\end{equation*}
while the flow on the center manifold is given by the reduced system (according to \cite[Theorem 2, Section 2.4]{Carr})
\begin{equation*}
\left\{\begin{array}{ll}\dot{x}&=-\frac{1}{\beta}x^2+x^2O(|(x,z)|),\\[1mm]
\dot{z}&=-\frac{1}{\beta}xz+xO(|(x,z)|^2),\end{array}\right.
\end{equation*}
We thus readily notice that the center manifolds behave as stable ones in a neighborhood of the point $Q_1$, which, together with the stable manifold generated by the only nonzero eigenvalue of $M(Q_1)$, give the conclusion.

\medskip

\textbf{2.} The linearization of the system \eqref{PSinf1} in a neighborhood of the critical point $Q_5$ has the matrix
$$
M(Q_5)=\left(
         \begin{array}{ccc}
           \frac{(1-m)(p-m)}{\sigma+2} & 0 & 0 \\[1mm]
           -1+\frac{N(p-m)}{\sigma+2} & \frac{p-m}{\sigma+2} & 0 \\[1mm]
           0 & 0 & -\frac{(p-m)(p-1)}{\sigma+2} \\
         \end{array}
       \right),
$$
and the conclusion follows readily by just inspecting the sign of the three eigenvalues of $M(Q_5)$ and their corresponding eigenvectors.

\medskip

\textbf{3.} At a formal level, one can see this fact in the following way: if there is an orbit connecting to some point $Q_{\gamma}$ in any possible way (I mean, belonging to
either a stable or unstable manifold), then along such orbit we have $z(\eta_1)\to\kappa\in(0,\infty)$ as either $\eta_1\to-\infty$ or $\eta_1\to\infty$, which, by undoing the changes of variable \eqref{change2} and \eqref{PSchange}, leads to profiles with the local behavior
\begin{equation}\label{interm5}
f(\xi)\sim C\xi^{-\sigma/(p-1)}, \qquad {\rm either \ as} \ \xi\to0 \ {\rm or \ as} \ \xi\to\infty,
\end{equation}
but plugging the ansatz \eqref{interm5} into the equation \eqref{ODE.forward} leads readily to a contradiction, as the dominating terms cannot be compensated in any of the two cases. A rigorous proof is rather tedious, but it follows analogous lines as the one performed for the similar critical point in \cite[Lemma 2.4]{ILS23} or \cite[Lemma 2.3]{IMS23}. Adapting those calculations we are led to the choice of $\kappa\in(0,\infty)$ for which orbits connecting to $Q_{\gamma}$ not contained in the invariant plane $\{x=0\}$ might exist, which is
$$
\frac{\sqrt{1-\gamma^2}}{\gamma}=\kappa=-\frac{L}{(p-1)(\sigma+2)}<0,
$$
hence there is no $Q_{\gamma}$ with $\gamma>0$ admitting such orbits.
\end{proof}
For the critical points $Q_2$ and $Q_3$, we employ the system obtained following the projection on the $Y$ variable according to \cite[Theorem 5(b), Section 3.10]{Pe}, namely
\begin{equation}\label{PSinf2}
\left\{\begin{array}{lll}\pm\dot{x}=-x-Nxw-\frac{p-m}{\sigma+2}x^2-x^2w-xzw,\\[1mm]
\pm\dot{z}=-pz-(N+\sigma)zw-\frac{p-m}{\sigma+2}xz-xzw-z^2w,\\[1mm]
\pm\dot{w}=-mw-(N-2)w^2-\frac{p-m}{\sigma+2}xw-xw^2-zw^2,\end{array}\right.
\end{equation}
where the new variables $x$, $z$, $w$ are obtained from the original variables of the system \eqref{PSsyst} as follows:
\begin{equation*}
x=\frac{X}{Y}, \qquad z=\frac{Z}{Y}, \qquad w=\frac{1}{Y}.
\end{equation*}
Thus, the flow of the system \eqref{PSsyst} in a neighborhood of the critical points $Q_2$ and $Q_3$ is topologically equivalent to the flow of the system \eqref{PSinf2} choosing the minus sign (for $Q_2$), respectively choosing the plus sign (for $Q_3$) in a neighborhood of the origin.
\begin{lemma}\label{lem.Q23}
The critical points $Q_2$ and $Q_3$ are, respectively, an unstable node and a stable node. The orbits going out of $Q_2$ correspond to profiles $f(\xi)$ such that there exists $\xi_0\in(0,\infty)$ and $\delta>0$ for which
\begin{equation*}
f(\xi_0)=0, \qquad (f^m)'(\xi_0)=C>0, \qquad f>0 \ {\rm on} \ (\xi_0,\xi_0+\delta),
\end{equation*}
while the orbits entering the stable node $Q_3$ correspond to profiles $f(\xi)$ such that there exists $\xi_0\in(0,\infty)$ and $\delta\in(0,\xi_0)$ for which
\begin{equation*}
f(\xi_0)=0, \qquad (f^m)'(\xi_0)=-C<0, \qquad f>0 \ {\rm on} \ (\xi_0-\delta,\xi_0).
\end{equation*}
\end{lemma}
The proof is rather easy and goes similarly to the one of \cite[Lemma 2.6]{IS21} to which we refer. We are thus left only with the critical point $Q_4$, whose analysis is done following the same recipe as in \cite[Section 4 and Appendix]{IMS23b}. On the one hand, the following technical result holds true:
\begin{lemma}\label{lem.Q4}
There are no solutions to Eq. \eqref{ODE.forward} satisfying simultaneously the following limits
\begin{equation*}
\begin{split}
&\xi^{\sigma}f(\xi)^{p-1}\to\infty, \qquad \xi^{\sigma+2}f(\xi)^{p-m}\to\infty, \qquad \xi^{\sigma-2}f(\xi)^{m+p-2}\to\infty, \\
&\xi^{\sigma+1}f(\xi)^{p-m+1}(f')^{-1}(\xi)\to\pm\infty,
\end{split}
\end{equation*}
the limits above being taken in any of the possible cases as $\xi\to0$, $\xi\to\xi_0\in(0,\infty)$ or $\xi\to\infty$.
\end{lemma}
Its proof is identical to the proof of \cite[Lemma 7.2]{IMS23b} and can be found in \cite[Appendix]{IMS23b}, since the fact that $m>m_c$ plays no role in the proof. On the other hand, letting in \eqref{PSinf1} the change of variable $w=xz$, we are left with the system
\begin{equation}\label{PSsyst3.ext}
\left\{\begin{array}{lll}\dot{x}=x[(m-1)y-2x],\\[1mm]
\dot{y}=-y^2-\frac{p-m}{\sigma+2}y-x-Nxy-w,\\[1mm]
\dot{w}=w[(\sigma-2)x+(m+p-2)y],\end{array}\right.
\end{equation}
and, according to the outcome of Lemma \ref{lem.Q4}, all the eventual trajectories connecting to $Q_4$ have $w(\eta_1)\to0$, hence they are found as trajectories connecting to the finite critical points of \eqref{PSsyst3.ext} with $w=0$. There are two such points, $Q_1'=(0,0,0)$ and $Q_5'=(0,-(p-m)/(\sigma+2),0)$. We have the following
\begin{lemma}\label{lem.Q4.bis}
If $m+p\geq2$, all the trajectories of the system \eqref{PSsyst3.ext} connecting to the critical points $Q_1'$ and $Q_5'$ are the same ones inherited from the critical points $Q_1$ and $Q_5$. If $m+p<2$, the point $Q_5'$ is an unstable node, while the point $Q_1'$ presents an unstable sector. In the latter case, the orbits going out of $Q_5'$ contain profiles with a vertical asymptote at some $\xi_0\in(0,\infty)$ and local behavior
\begin{equation*}
f(\xi)\sim\left[\frac{\beta(1-m)}{2m}\xi^2-K\right]^{-1/(1-m)}, \qquad {\rm as} \ \xi\to\xi_0=\sqrt{\frac{2mK}{\beta(1-m)}}, \ \xi>\xi_0,
\end{equation*}
while the orbits going out of $Q_1'$ on the unstable sector present a vertical asymptote at some $\xi_0\in(0,\infty)$ with the local behavior given by one of the following alternatives, all them taken as $\xi\to\xi_0$, $\xi>\xi_0$:
\begin{equation}\label{beh.Q11f}
f(\xi)\sim\left\{\begin{array}{lll}\left[\frac{p-1}{\beta\sigma}\xi^{\sigma}-K\right]^{-1/(p-1)}, & K=\frac{p-1}{\beta\sigma}\xi_0^{\sigma}, & {\rm if} \ \sigma>0,\\ \left[\frac{p-1}{\beta}\ln\,\xi-K\right]^{-1/(p-1)}, & K=\frac{p-1}{\beta}\ln\,\xi_0, & {\rm if} \ \sigma=0, \\
\left[K-\frac{p-1}{|\sigma|\beta}\xi^{\sigma}\right]^{-1/(p-1)}, & K=\frac{p-1}{|\sigma|\beta}\xi_0^{\sigma}, & {\rm if} \ \sigma\in(-2,0).\end{array}\right.
\end{equation}
\end{lemma}
\begin{proof}[Sketch of the proof]
For the point $Q_5'$, the only difference with respect to \cite[Lemma 4.2]{IMS23b} is the direction of the flow, which is readily obtained by computing the matrix of the linearization in the system \eqref{PSsyst3.ext}. With respect to the point $Q_1'$, the proof is much more involved, as it is a non-hyperbolic critical point with different sectors, but the details can be found in the proof of \cite[Lemma 4.3]{IMS23b}. As a noticeable difference with respect to the previously quoted proof, observe that
$$
t=\frac{p-m}{\sigma+2}y+x+w
$$
is the change of variable in \eqref{PSsyst3.ext} in order to apply the center manifold theorem, which in the end leads to
\begin{equation}\label{interm13}
\frac{p-m}{\sigma+2}y=-w+o(w)
\end{equation}
on the unstable sector, and passing to profiles we obtain
$$
(f^{1-p})'(\xi)\sim\frac{p-1}{\beta}\xi^{\sigma-1}, \qquad {\rm as} \ \xi\to\xi_0\in(0,\infty), \ \xi>\xi_0,
$$
which leads the alternatives in \eqref{beh.Q11f} by integration on $(\xi_0,\xi)$. We omit the rest of the details, as they are given in the proof of \cite[Lemma 4.3]{IMS23b}.
\end{proof}
These two critical points will not play any specific role in the forthcoming analysis, but we have just classified all the possible local behaviors of global (in time) self-similar solutions.

\medskip

\noindent \textbf{The critical case $m=m_c$.} Notice that, if $m=m_c$, we have
\begin{equation}\label{interm28}
p_L(\sigma)=p_F(\sigma)=p_c(\sigma)=\frac{N+\sigma}{N},
\end{equation}
and we are thus by default in the range $p>p_c(\sigma)$. By inspecting the previous analysis, we remark that the only difference that appears for $m=m_c$ with respect to the critical points of the system \eqref{PSsyst} is the fact that $P_1$ and $P_3$ coincide.
\begin{lemma}\label{lem.P1P3crit}
Let $m=m_c$ and $p>p_L(\sigma)=p_c(\sigma)$. Then the critical point $P_1=P_3$ in the system \eqref{PSsyst} is a non-hyperbolic critical point which behaves as a saddle, with a two-dimensional center-stable manifold and a one-dimensional unstable manifold. The unstable manifold is contained in the invariant $Y$-axis, while the trajectories entering $P_1=P_3$ on the center-stable manifold contain profiles with the decay as $\xi\to\infty$ given by \eqref{beh.P3crit}.
\end{lemma}
Notice that $-N=-2/(1-m)$ if $m=m_c$, hence \eqref{beh.P3crit} is a similar decay to \eqref{beh.P3}, but adding up a logarithmic correction, sometimes specific to critical cases.
\begin{proof}
The linearization of the system \eqref{PSsyst} in a neighborhood of the critical point has the matrix
$$
M(P_1)=\left(
         \begin{array}{ccc}
           0 & 0 & 0 \\[1mm]
           \frac{N(p-p_c(\sigma))}{\sigma+2} & N-2 & -1 \\[1mm]
           0 & 0 & N(p_c(\sigma)-p) \\
         \end{array}
       \right),
$$
with eigenvalues
$$
\lambda_1=0, \qquad \lambda_2=N-2>0, \qquad \lambda_3=N(p_c(\sigma)-p)<0.
$$
In order to analyze the flow on the one-dimensional center manifold, we perform the translation of $P_1$ to the origin by setting $W=Y+(N-2)/m$. Taking into account that $m=m_c$, the system \eqref{PSsyst} writes
\begin{equation}\label{syst.P1crit}
\left\{\begin{array}{ll}\dot{X}=\frac{2}{N}XW,\\[1mm]
\dot{W}=\frac{N(p-p_c(\sigma))}{\sigma+2}X+(N-2)W-Z-\frac{N-2}{N}W^2-\frac{Np-N+2}{N(\sigma+2)}XW,\\[1mm]
\dot{Z}=-N(p-p_c(\sigma))Z+\frac{Np-N+2}{N}ZW.\end{array}\right.
\end{equation}
According to \cite[Theorem 3.2.1]{GH}, the center manifold is tangent to the eigenvector corresponding to the zero eigenvalue, thus we can conclude that a first order approximation of it in terms of $W$ and $X$ writes in a neighborhood of $P_1$ as
\begin{equation}\label{interm29}
Y+\frac{N-2}{m}=W=-\frac{N(p-p_c(\sigma))}{(N-2)(\sigma+2)}X+o(X).
\end{equation}
More tedious calculation can give a more precise approximation of the center manifold, but the previous approximation is enough to conclude, by using the reduction principle \cite[Theorem 2, Section 2.4]{Carr} and the first equation in \eqref{syst.P1crit}, that the direction of the flow on the center manifolds is given by the equation
$$
\dot{X}=-\frac{2(p-p_c(\sigma))}{(N-2)(\sigma+2)}X^2<0,
$$
hence any center manifold has a stable flow. Together with the orbit on the stable manifold corresponding to the negative eigenvalue $\lambda_3$, we thus find a two-dimensional center-stable manifold. The uniqueness of the unstable manifold \cite[Theorem 3.2.1]{GH} together with the direction of the eigenvector $e_2=(0,1,0)$ associated to the eigenvalue $\lambda_2=N-2$ and the invariance of the $Y$-axis ensure that the unstable manifold of $P_1$ is fully contained in the $Y$-axis, as claimed. We are left with the local behavior of the profiles on the center-stable manifold. To this end, we begin with the expression of the center manifold and give below a formal deduction of the decay \eqref{beh.P3crit}. By replacing \eqref{interm29} in terms of profiles, we get
\begin{equation}\label{interm30}
\frac{\xi f'(\xi)}{f(\xi)}\sim-\frac{2}{1-m}-\frac{(Np-N-\sigma)\alpha}{m(N-2)(\sigma+2)}\xi^2f(\xi)^{1-m}, \qquad {\rm as} \ \xi\to\infty.
\end{equation}
We plug in \eqref{interm30} the ansatz
\begin{equation}\label{ansatz}
f(\xi)=\xi^{-2/(1-m)}g(\xi)=\xi^{-N}g(\xi)
\end{equation}
and deduce after straightforward calculations the differential equation satisfied in a first approximation by $g(\xi)$ as
\begin{equation}\label{interm31}
\frac{\xi g'(\xi)}{g(\xi)}=-\frac{N^2}{2(N-2)^2}g(\xi)^{1-m},
\end{equation}
which by integration leads to
$$
g(\xi)=\left[\frac{N}{(N-2)^2}\ln\,\xi\right]^{-1/(1-m)}=\left[\frac{N}{(N-2)^2}\ln\,\xi\right]^{-N/2},
$$
and the decay \eqref{beh.P3crit} follows from the ansatz \eqref{ansatz}. In order to make the above rigorous, we actually need the next order of approximation of the center manifold \eqref{interm29} in terms of $X^2$, that can be obtained following \cite[Theorem 3, Section 2.5]{Carr}, in order to show that, when undoing the change of variable and obtain \eqref{interm31}, we can safely integrate the equation as the next term of approximation (the one coming from the term in $X^2$) is indeed of lower order even when dividing by $\xi$ (recall that $\xi\to\infty$). We omit here these technical details which require a few more calculations.
\end{proof}

\section{Proof of Theorem \ref{th.global.super}}\label{sec.global.super}

We complete in this section the proof of Theorem \ref{th.global.super}. To this end, let us recall as an outcome of \eqref{interm1} and \eqref{interm2} that the two-dimensional stable manifold of $P_0$ can be described as the one-parameter family of trajectories
\begin{equation}\label{orbP0.global}
(l_C): \qquad Y(\eta)\sim-\frac{X(\eta)}{N}-\frac{CX(\eta)^{(\sigma+2)/2}}{N+\sigma}, \qquad C\in[0,\infty), \qquad {\rm as} \ \eta\to-\infty,
\end{equation}
together with the orbit $l_{\infty}$ included in the invariant plane $\{X=0\}$. With respect to the latter, we notice that the invariant plane $\{X=0\}$ of the system \eqref{PSsyst} is exactly the same one as in \cite[Section 5]{IMS23b} and we conclude from the analysis therein that the orbit $l_{\infty}$

$\bullet$ connects to the stable node $Q_3$ if $p\in(p_L(\sigma),p_s(\sigma))$.

$\bullet$ connects to the point $P_1$ if $p=p_s(\sigma)$, with the explicit trajectory
\begin{equation}\label{cylinder}
Z=-\frac{N+\sigma}{N-2}(mY+N-2)Y.
\end{equation}

$\bullet$ connects to the point $P_2$ if $p>p_s(\sigma)$.

We analyze now the other limit trajectory, namely $l_0$, corresponding to $C=0$ in \eqref{orbP0.global} and thus fully contained in the invariant plane $\{Z=0\}$. It is at this point where the Fujita exponent comes into play with decisive effect in the analysis.
\begin{lemma}\label{lem.Z0}
Let $N\geq1$, $\sigma>\max\{-2,-N\}$, $m\geq m_c$ and $p>p_L(\sigma)$. Then the orbit $l_0$ contained in the invariant plane $\{Z=0\}$ connects to the stable node $Q_3$ if $p<p_F(\sigma)$, to the saddle point $P_3$ if $p=p_F(\sigma)$ and to the non-hyperbolic point $Q_1$ which behaves as a stable node if $p>p_F(\sigma)$.
\end{lemma}
\begin{proof}
Assume first that $m>m_c$. The system \eqref{PSsyst} reduces in the plane $\{Z=0\}$ to
\begin{equation}\label{systZ0}
\left\{\begin{array}{ll}\dot{X}=X(2+(1-m)Y),\\\dot{Y}=-X-(N-2)Y-mY^2-\frac{p-m}{\sigma+2}XY.\end{array}\right.
\end{equation}
Consider now the line $Y=-X/N$. An easy calculation gives that the crossing point between this line and the isocline $Y=-2/(1-m)$ is attained at $X_0=2N/(1-m)$ and with respect to the critical point $P_3$, this intersection point satisfies
\begin{equation}\label{interm6}
X(P_3)-X_0=\frac{2(\sigma+2)(mN-N+2)}{L(1-m)}-\frac{2N}{1-m}=\frac{4N(p_F(\sigma)-p)}{L(1-m)}.
\end{equation}
We thus deduce from \eqref{interm6} that $Y=-X/N$ is the straight line connecting $P_0$ and $P_3$ if $p=p_F(\sigma)$, while the crossing point lies below $P_3$ (that is, smaller coordinate $X_0$) if $p<p_F(\sigma)$ and above $P_3$ if $p>p_F(\sigma)$. Moreover, one can readily check that if $p=p_F(\sigma)$, $Y=-X/N$ is also a trajectory of the system \eqref{systZ0}. Taking the normal vector $\overline{n}=(1/N,1)$, the flow of the system \eqref{systZ0} on the line $Y=-X/N$ is given by the sign of the scalar product between the vector field of the system and this normal vector, which gives the expression
\begin{equation}\label{interm7}
F(X)=-\frac{(mN-Np+\sigma+2)X^2}{N^2(\sigma+2)}=\frac{X^2}{N(\sigma+2)}(p-p_F(\sigma)).
\end{equation}
We are only left with finding the region in which the orbit $l_0$ goes out of $P_0$, with respect to the line $Y=-X/N$, knowing already that they coincide if $p=p_F(\sigma)$. To this end, we need to compute the second order of the local expansion of the orbit $l_0$ in a neighborhood of $(X,Y)=(0,0)$. One can proceed as in \cite[Section 2.7]{Shilnikov} to identify the Taylor expansion of the unstable manifold in a neighborhood of a saddle point. In our case, we set
\begin{equation}\label{interm8}
Y=\psi(X):=-\frac{1}{N}X+KX^2+o(X^2), \qquad K\in\real,
\end{equation}
and plug the ansatz \eqref{interm8} into the system \eqref{systZ0}. We thus have
$$
\dot{Y}=\psi'(X)\dot{X}=\left(-\frac{1}{N}+2KX+o(X)\right)\dot{X},
$$
which after equating the coefficients of $X$ and some straightforward calculations leads to the expansion
\begin{equation}\label{interm9}
Y(\eta)\sim-\frac{1}{N}X(\eta)+\frac{p-p_F(\sigma)}{N(N+2)(\sigma+2)}X^2(\eta)+o(X^2(\eta)), \qquad {\rm as} \ \eta\to-\infty.
\end{equation}
From the expansion \eqref{interm9}, the direction of the flow on the line $Y=-X/N$ given by the sign of $F(X)$ in \eqref{interm7}, and the position of the crossing between the line $Y=-X/N$ with the vertical $Y=-2/(1-m)$ of the critical point $P_3$, we deduce that, if $p<p_F(\sigma)$, the orbit $l_0$ goes out and then stays forever in the region $\{X>0, Y<-X/N\}$, as it cannot cross the line $Y=-X/N$, while the point $P_3$ belong to the opposite half-plane, while if $p>p_F(\sigma)$, the orbit $l_0$ goes out and then stays forever in the region $\{X>0, Y>-X/N\}$ as it cannot cross the line $Y=-X/N$, while the point $P_3$ belongs to the opposite region. Poincar\'e-Bendixon's theory \cite[Theorem 5, Section 3.7]{Pe} and the non-existence of critical points in the two half-planes in any of the two cases (as $P_3$ lies in the opposite one) imply that there are no periodic orbits or other type of $\omega$-limits, thus $l_0$ has to connect to a critical point. If $p<p_F(\sigma)$, this is obviously $Q_3$ since $X(\eta)$ is decreasing after crossing the line $Y=-2/(1-m)$. If $p>p_F(\sigma)$, the orbit remains forever in the region
$$
\mathcal{S}=\left\{X>0, Y>-\frac{1}{N}X, -\frac{2}{1-m}<Y<0\right\},
$$
since the flow on the line $Y=-2/(1-m)$ has positive sign if $X>X(P_3)$ and the flow on the line $\{Y=0\}$ points obviously into the negative half-plane. It follows that $\dot{X}(\eta)>0$ for any $\eta\in\real$, hence $X(\eta)$ increases with $\eta$. This monotonicity, together with the non-existence of any finite critical point in the region $\mathcal{S}$ if $p>p_F(\sigma)$ imply that $X(\eta)\to\infty$ as $\eta\to\infty$, hence also $Y(\eta)/X(\eta)\to0$ as $\eta\to\infty$. We then readily infer that the orbit $l_0$ enters the critical point $Q_1$. In the critical case $m=m_c$, we are by default in the case $p>p_F(\sigma)=p_L(\sigma)$ according to \eqref{interm28} and the above considerations corresponding to the range $p>p_F(\sigma)$ still hold true identically.
\end{proof}
We are now ready to complete the proof of Theorem \ref{th.global.super}.
\begin{proof}[Proof of Theorem \ref{th.global.super}]
\textbf{(a) Range $p_L(\sigma)<p\leq p_F(\sigma)$.} Let us consider the plane $\{Y=-X/N\}$, and observe that the direction of the flow of the system \eqref{PSsyst} on this plane is given by the sign of the expression
$$
G(X,Z)=\frac{X^2}{N(\sigma+2)}(p-p_F(\sigma))-Z<0,
$$
hence, the region $\mathcal{V}:=\{X>0, Y\leq-X/N\}$ is positively invariant for the flow. Moreover, \eqref{orbP0.global} and Lemma \ref{lem.Z0} imply that all the orbits $l_C$ of the stable manifold of $P_0$ go out into the region $\mathcal{V}$, with the unique exception of the orbit $l_0$ if $p=p_F(\sigma)$. This implies that no other orbit $l_C$ enters either $P_3$ or $Q_1$, the critical points encoding a tail behavior as $\xi\to\infty$ (and in reality, one can rather easily show that all these orbits connect to the stable node $Q_3$, although we do not need this information for the proof). Thus, there are no profiles with any tail as $\xi\to\infty$, proving Part 3 in Theorem \ref{th.global.super}.

\medskip

\textbf{(b) Range $p_F(\sigma)<p<p_s(\sigma)$.} This is the most interesting range in the theorem, since it gives rise to a specific behavior to the supercritical fast diffusion. To prove it, we introduce the following three sets:
\begin{equation*}
\begin{split}
&\mathcal{A}=\{C\in(0,\infty): {\rm the \ orbit} \ l_C \ {\rm enters \ the \ point} \ Q_3\},\\
&\mathcal{B}=\{C\in(0,\infty): {\rm the \ orbit} \ l_C \ {\rm does \ not \ enter} \ Q_3 \ {\rm nor} \ Q_1\},\\
&\mathcal{C}=\{C\in(0,\infty): {\rm the \ orbit} \ l_C \ {\rm enters \ the \ point} \ Q_1\}.
\end{split}
\end{equation*}
Since, by Lemma \ref{lem.Q23}, respectively Lemma \ref{lem.Q1Q5Qg}, both points $Q_3$ and $Q_1$ behave like attractors for the system \eqref{PSsyst}, we infer that the sets $\mathcal{A}$ and $\mathcal{C}$ are open. Moreover, the analysis in \cite[Section 5]{IMS23b} proves that the orbit $l_{\infty}$ connects to $Q_3$ and by continuity and the stability of $Q_3$, the set $\mathcal{A}$ is nonempty and in fact contains an open interval $(C^*,\infty)$. Lemma \ref{lem.Z0} shows that in this range the orbit $l_0$ enters $Q_1$ and the stability of $Q_1$ and continuity arguments give that the set $\mathcal{C}$ is also nonempty and in fact contains an interval of the form $(0,C_*)$. We obtain by standard topology that the set $\mathcal{B}$ is nonempty. Let us fix now $C_0\in\mathcal{B}$. We want to show that the orbit $l_{C_0}$ enters the saddle point $P_3$ if $m>m_c$ or the non-hyperbolic point $P_1=P_3$ if $m=m_c$.

Assume that $l_{C_0}$ reaches an $\omega$-limit set $\Gamma\subset\real^3$ contained in the octant $\{X\geq0, Y<0, Z\geq0\}$. Introduce the function
$$
g(\xi):=\xi^{(\sigma+2)/(p-m)}f(\xi).
$$
One can then straightforwardly compute, starting from \eqref{ODE.forward}, the differential equation satisfied by $g$, which is
\begin{equation}\label{interm11}
\begin{split}
\xi^2(g^m)''(\xi)&+\left(N-1-\frac{2m(\sigma+2)}{p-m}\right)\xi(g^m)'(\xi)+\frac{m(\sigma+2)(m\sigma+m+p)}{(p-m)^2}g^m(\xi)\\
&+\frac{p-m}{L}\xi^{(m-1)(\sigma+2)/(p-m)+3}g'(\xi)+g^p(\xi)=0.
\end{split}
\end{equation}
We readily notice that $g(\xi)$ solution to \eqref{interm11} cannot admit minimum points, thus either $g$ is monotone or it has a single maximum points and becomes then decreasing, and in both cases there exists
$$
l:=\lim\limits_{\xi\to\infty}g(\xi)=\lim\limits_{\xi\to\infty}(mZ(\xi))^{1/(p-m)},
$$
whence $Z(\xi)$ also has a limit as $\xi\to\infty$. If $Z(\xi)\to0$ as $\xi\to\infty$ along the orbit $l_{C_0}$, we infer that $\Gamma$ is contained in the invariant plane $\{Z=0\}$. Poincar\'e-Bendixon's theory \cite[Section 3.7]{Pe} applies then to the trajectory $\Gamma$, which is also invariant for \eqref{PSsyst} according to \cite[Theorem 2, Section 3.2]{Pe}, and readily gives that $\Gamma$ is a periodic orbit or a critical point. We then infer from \cite[Theorem 5, Section 3.7]{Pe} that $P_3$ lies inside the region limited by $\Gamma$, and this implies that $\Gamma=\{P_3\}$, since $P_3$ is a hyperbolic critical point in the plane $\{Z=0\}$. Assume next that $Z(\xi)\to L_0\in(0,\infty)$ as $\xi\to\infty$ along the trajectory $l_{C_0}$. Since $\Gamma$ itself is an invariant trajectory of the system, it follows from the third equation of the system \eqref{PSsyst} that on $\Gamma$ we must have $Y(\eta)\to-(\sigma+2)/(p-m)$ as $\eta\to\infty$. The second equation of the system \eqref{PSsyst} then gives that
$$
-(N-2)Y(\eta)-mY(\eta)^2-Z(\eta)=0, \qquad {\rm that \ is} \ L_0=\frac{(N-2)(\sigma+2)(p-p_c(\sigma))}{(p-m)^2},
$$
while the first equation implies that either $X(\eta)=0$ for any $\eta\in\real$ (in which case $\Gamma=\{P_2\}$) or $\dot{X}(\eta)>0$ for any $\eta\in\real$. The latter gives $X(\eta)\to\infty$ as $\eta\to\infty$ and thus $\Gamma=\{Q_1\}$ in variables $(x,y,z)$ defined in \eqref{change2}. In both cases, we reach a contradiction: $\Gamma=\{P_2\}$ implies $C_0=\infty$ and $\Gamma=\{Q_1\}$ implies $C_0\in\mathcal{C}$, thus not in $\mathcal{B}$. Finally, if $Z(\xi)\to\infty$ as $\xi\to\infty$ (which is the same for $Z(\eta)$ as $\eta\to\infty$) along the orbit $l_{C_0}$, we conclude that there exists $\eta_0\in\real$ such that $\dot{Z}(\eta)>0$ for $\eta>\eta_0$. The third equation in \eqref{PSsyst} then gives
\begin{equation}\label{interm12}
-\frac{\sigma+2}{p-m}\leq Y(\eta)<0, \qquad \eta>\eta_0,
\end{equation}
which in particular also gives $\dot{X}(\eta)>0$ for $\eta>\eta_0$. This monotonicity of both $X(\eta)$ and $Z(\eta)$ along the trajectory $l_{C_0}$, together with the boundedness of $Y(\eta)$ following from \eqref{interm12}, shows that $\Gamma$ is a set at infinity composed by points with coordinates $x=y=0$ in the variables defined by \eqref{change2}, which means that $\Gamma$ is at least a part of the critical line composed by the points $Q_{\gamma}$ together with $Q_1$ and $Q_4$. But this is impossible, since Lemmas \ref{lem.Q4}, \ref{lem.Q4.bis} and \ref{lem.Q1Q5Qg} imply that there cannot be any trajectory entering or approaching any of $Q_{\gamma}$ with $\gamma>0$ and $Q_4$ from the finite part of the phase space, while $Q_1$ behaves like an attractor and in this case, we would have $C_0\in\mathcal{C}$. We have thus proved that the only possibility is $\Gamma=\{P_3\}$. Thus, for any $C\in\mathcal{B}$, the trajectories $l_C$ connect $P_0$ to $P_3$. Lemmas \ref{lem.P0} and \ref{lem.P3}, together with the non-emptiness of the set $\mathcal{C}$, complete the proof of Part 1 in Theorem \ref{th.global.super}.

\medskip

\textbf{(c) Range $p_s(\sigma)\leq p<\infty$.} We consider in this range the cylinder with basis on the explicit curve \eqref{cylinder}. The flow of the system \eqref{PSsyst} over this cylinder is given by the sign of the expression
\begin{equation}\label{flow.cyl.ext}
\begin{split}
E(X,Y;p)=\frac{N+\sigma}{N-2}&\left[(p_s(\sigma)-p)Y^2(mY+N-2)\right.\\&\left.-X\left(1+\frac{p-m}{\sigma+2}Y\right)(2mY+N-2)\right],
\end{split}
\end{equation}
which is obviously negative if $p=p_s(\sigma)$, since
$$
E(X,Y;p_s(\sigma))=-(N+\sigma)X\left(1+\frac{2m}{N-2}Y\right)^2<0,
$$
while if $p>p_s(\sigma)$ it might only be positive on a subset of the interval
\begin{equation}\label{interm10}
-\frac{N-2}{m}<Y<-\frac{\sigma+2}{p-m}.
\end{equation}
But on an orbit lying in the interior of the cylinder, we notice that $\dot{Z}(\eta)<0$ if $Y(\eta)$ satisfies \eqref{interm10}, while $Z$ still increases on the boundary of the cylinder if $Y$ satisfies \eqref{interm10}. It follows that no trajectory may leave the interior of the cylinder, once there. It is then easy to see that, if $p\geq p_s(\sigma)$, all the orbits $l_C$ as in \eqref{orbP0.global} go out in the interior of the cylinder (with the exception of $l_{\infty}$ if $p=p_s(\sigma)$, in such case $l_{\infty}$ being exactly the basis of the cylinder in the plane $\{X=0\}$). Hence, all the orbits $(l_C)$, with $C\in(0,\infty)$, remain forever in the interior of the cylinder. Since for $m>m_c$ we have
$$
\frac{2}{1-m}-\frac{N-2}{m}=\frac{mN-N+2}{m(1-m)}>0,
$$
we infer that $P_3$ lies in the exterior of the cylinder \eqref{cylinder} and thus there are no trajectories connecting $P_0$ and $P_3$. If $m=m_c$, the point $P_3=P_1$ belongs to the surface of the cylinder, but, as we have seen in Lemma \ref{lem.P1P3crit} together with \cite[Lemma 1, Section 2.4]{Carr}, the orbits on the center-stable manifold enter the point tangent to the center manifold \eqref{interm29} and thus come from the exterior of the cylinder \eqref{cylinder}, and we conclude that there are no trajectories connecting $P_0$ to $P_3=P_1$ also in this case. Lemma \ref{lem.Z0}, together with the fact that $Q_1$ behaves like an attractor for the orbits of \eqref{PSsyst}, imply that there are infinitely many trajectories between $P_0$ and $Q_1$, proving thus Part 2 in Theorem \ref{th.global.super}.
\end{proof}
We end this section by plotting in Figure \ref{fig2} the outcome of a numerical experiment to visualize how the orbits stem from $P_0$ and connect to either $P_3$ or critical points at infinity, agreeing with the above proof in the range $p_F(\sigma)=5/3<p<p_s(\sigma)=14/3$ when the specific orbits connecting $P_0$ to $P_3$ are obtained.
\begin{figure}[ht!]
  \begin{center}
  \includegraphics[width=11cm,height=7.5cm]{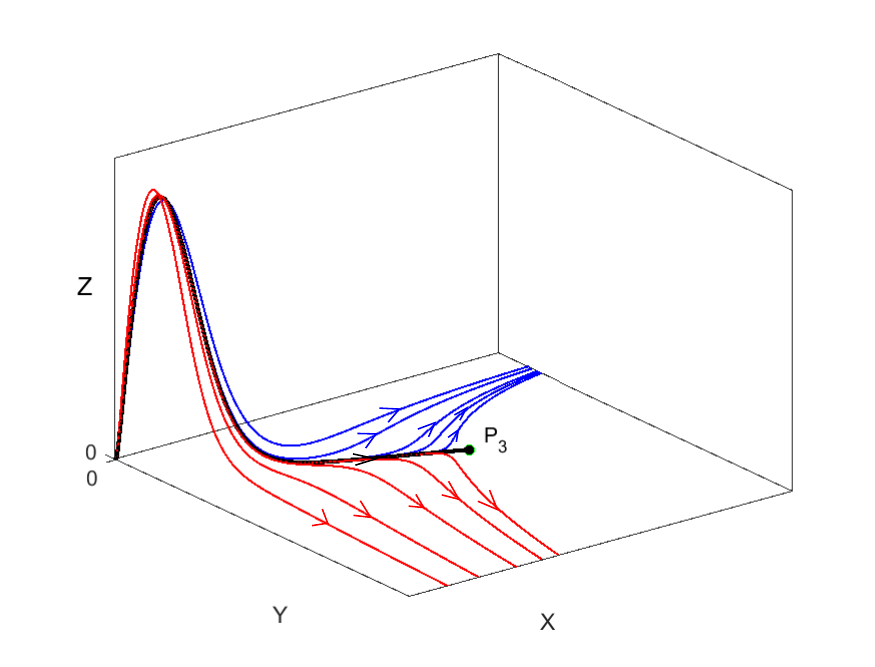}
  \end{center}
  \caption{Phase portrait with orbits on the unstable manifold of $P_0$. Experiment for $m=2/3$, $p=3$, $\sigma=1$, $N=3$.}\label{fig2}
\end{figure}

\section{Blow-up profiles. General facts}\label{sec.bu.general}

We go back now to self-similar solutions in the form \eqref{backward.SS}, presenting finite time blow-up at time $T\in(0,\infty)$ and whose profiles solve the differential equation \eqref{ODE.backward}. In order to classify the possible behaviors, we consider the same change of variables \eqref{PSchange} and obtain the following system
\begin{equation}\label{PSsyst.bu}
\left\{\begin{array}{ll}\dot{X}=X(2+(1-m)Y), \\ \dot{Y}=X-(N-2)Y-Z-mY^2+\frac{p-m}{\sigma+2}XY, \\ \dot{Z}=Z(\sigma+2+(p-m)Y).\end{array}\right.
\end{equation}
Fixing ourselves again in the supercritical range $m\in[m_c,1)$ and in dimension $N\geq3$ for simplicity, we observe that the system \eqref{PSsyst.bu} only has three finite critical points
$$
P_0=(0,0,0), \ \ P_1=\left(0,-\frac{N-2}{m},0\right), \ \ P_2=\left(0,-\frac{\sigma+2}{p-m},\frac{(N-2)(\sigma+2)(p-p_c(\sigma))}{(p-m)^2}\right),
$$
the latter existing only if $p>p_c(\sigma)$. The local analysis of the system in a neighborhood of them is rather similar to the one performed in Section \ref{sec.local} and detailed calculations are given exactly for this system in \cite{IMS23b}, thus we only give below the differences with respect to the previous analysis:

$\bullet$ the critical point $P_0$ has a two-dimensional unstable manifold and a one-dimensional stable manifold as in Lemma \ref{lem.P0}. It is easy to see that \eqref{interm1} holds true and, by integration, the orbits contained in the unstable manifold of $P_0$ have now the local expansion
\begin{equation}\label{orbP0.bu}
(r_C): \qquad Y(\eta)\sim\frac{X(\eta)}{N}-\frac{CX(\eta)^{(\sigma+2)/2}}{N+\sigma}, \qquad C\in[0,\infty), \qquad {\rm as} \ \eta\to-\infty,
\end{equation}
together with the orbit $r_{\infty}$ included in the invariant plane $\{X=0\}$. We observe that, if $\sigma>0$, then $(\sigma+2)/2>0$ and all the orbits $r_C$ with $C\in(0,\infty)$ go out into the positive half-space $\{Y>0\}$ since the first order of approximation becomes
$$
Y(\eta)\sim\frac{X(\eta)}{N}+o(X(\eta)), \qquad {\rm as} \ \eta\to-\infty.
$$
If $\sigma<0$, we conclude that $\sigma+2<2$, thus the first order of approximation is given by
$$
Y(\eta)\sim-C\frac{X(\eta)^{(\sigma+2)/2}}{N+\sigma}+o(X(\eta)^{(\sigma+2)/2}), \qquad {\rm as} \ \eta\to-\infty,
$$
hence on the one hand, all the orbits $r_C$, $C\in[0,\infty)$, go out into the negative half-space $\{Y<0\}$ and on the other hand, by undoing \eqref{PSchange} on the previous approximation and an integration on $(0,\xi)$ for $\xi>0$ small, we find the local behavior \eqref{beh.P0f}. Finally, if $\sigma=0$, we infer from \eqref{orbP0.bu} that both terms in the approximation have the same strength, and the way the trajectories go out from the origin depends on the value of the constant $C>0$: if $C>1$ then $r_C$ enters the half-space $\{Y<0\}$ and if $C<1$ then $r_C$ enters the half-space $\{Y>0\}$.

$\bullet$ The local analysis near the critical points $P_1$ and $P_2$ is totally identical to the one done in Lemmas \ref{lem.P1} and \ref{lem.P2}.

$\bullet$ The critical point $Q_1$ is now identified as $Q_1=(0,0,0)$ in the system (deduced with the same change of variable \eqref{change2})
\begin{equation}\label{PSinf1.bu}
\left\{\begin{array}{ll}\dot{x}=x[(m-1)y-2x],\\
\dot{y}=-y^2+\frac{p-m}{\sigma+2}y+x-Nxy-xz,\\
\dot{z}=z[(p-1)y+\sigma x],\end{array}\right.
\end{equation}
hence the analysis on the two-dimensional center manifold is the same as in Lemma \ref{lem.Q1Q5Qg} and the flow on it has stable direction, while the non-zero eigenvalue is now positive. Thus the center manifold is now unique (see \cite[Theorem 3.2']{Sij}) and its trajectories still contain profiles with tail \eqref{beh.Q1f} as $\xi\to\infty$. In particular, we will be interested in the equation of the center manifold of $Q_1$. By setting $w=x+(p-m)y/(\sigma+2)$ in the system \eqref{PSinf1.bu} and then applying \cite[Theorem 3, Section 2.5]{Carr} in order to identify the Taylor expansion up to second order, we obtain the following equation for the center manifold:
\begin{equation}\label{center.Q1}
y=-\frac{\sigma+2}{p-m}x+\frac{(\sigma+2)^2(N-2)(p_c(\sigma)-p)}{(p-m)^3}x^2+\frac{\sigma+2}{p-m}xz+o(|(x,z)|^2).
\end{equation}
More details of the deduction of the equation \eqref{center.Q1} are given in \cite[Lemma 3.1]{IMS23b}.

$\bullet$ The critical point $Q_5$ is now identified as $Q_5=(0,(p-m)/(\sigma+2),0)$ in the system \eqref{PSinf1.bu}, and we readily obtain by inspecting the matrix of the linearization of the system (see \cite[Lemma 3.2]{IMS23b}) that it has a two-dimensional stable manifold completely included in the invariant plane $\{z=0\}$ (and also $\{Z=0\}$) and a one-dimensional unstable manifold fully included in the invariant plane $\{x=0\}$.

$\bullet$ The local analysis near $Q_2$ and $Q_3$ is totally identical to the one in Lemma \ref{lem.Q23}.

$\bullet$ With respect to the family of critical points $Q_{\gamma}$, one can readily check (by repeating the calculations in, for example, \cite[Lemma 2.4]{ILS23}) that there exists a single point of the family allowing for trajectories of the system \eqref{PSsyst.bu} arriving to it from the finite part of the phase space, namely $Q_{\gamma}$ with
$$
\frac{\sqrt{1-\gamma^2}}{\gamma}=\kappa=\frac{L}{(\sigma+2)(p-1)}>0.
$$
The analysis of the center manifold with similar calculations as in \cite[Lemma 2.4]{ILS23} gives that there exists a unique orbit entering this point and containing profiles with the local behavior
\begin{equation}\label{beh.Qg}
f(\xi)\sim\left(\frac{1}{p-1}\right)^{1/(p-1)}\xi^{-\sigma/(p-1)}, \qquad {\rm as} \ \xi\to\infty.
\end{equation}
Notice that \eqref{beh.Qg} is a tail behavior only for $\sigma>0$, and in this case, it is the slowest decay of a tail as $\xi\to\infty$, since, for example,
$$
\frac{\sigma}{p-1}-\frac{\sigma+2}{p-m}=\frac{-L}{(p-m)(p-1)}<0.
$$

$\bullet$ The critical point $Q_4$ splits again into the points $Q_1'$ and $Q_5'$ if setting $w=xz$ in the system \eqref{PSinf1.bu}. The local analysis of them is very similar to the one in \cite[Section 4]{IMS23b}. Gathering the results therein, new orbits appear only if $m+p<2$ (similarly as in Lemma \ref{lem.Q4.bis}), as for $m+p\geq2$ we get exactly  the same orbits of $Q_1$ and $Q_5$. In the former case, we have
\begin{lemma}\label{lem.Q4.bu}
Assume that $m+p<2$. Then the critical point $Q_5'$ is a stable node, and the profiles contained in the trajectories approaching it present a vertical asymptote as $\xi\to\xi_0\in(0,\infty)$, $\xi<\xi_0$, with the local behavior
\begin{equation*}
f(\xi)\sim\left[K-\frac{(1-m)\beta}{2m}\xi^2\right]^{-1/(1-m)}, \qquad K=\frac{(1-m)\beta}{2m}\xi_0^2,
\end{equation*}
while the orbits entering the critical point $Q_1'$ on the two-dimensional stable sector which differs from the one of $Q_1$ have the local behavior given by
\begin{equation}\label{beh.Q41.bu}
f(\xi)\sim\left\{\begin{array}{lll}\left[K-\frac{p-1}{\beta\sigma}\xi^{\sigma}\right]^{-1/(p-1)}, & K=\frac{p-1}{\beta\sigma}\xi_0^{\sigma}, & {\rm if} \ \sigma>0,\\ \left[K-\frac{p-1}{\beta}\ln\,\xi\right]^{-1/(p-1)}, & K=\frac{p-1}{\beta}\ln\,\xi_0, & {\rm if} \ \sigma=0, \\
\left[\frac{p-1}{|\sigma|\beta}\xi^{\sigma}\right]^{-1/(p-1)}, & & {\rm if} \ \sigma\in(-2,0).\end{array}\right.
\end{equation}
as either $\xi\to\xi_0\in(0,\infty)$, $\xi<\xi_0$ if $\sigma\geq0$, or $\xi\to\infty$ if $\sigma\in(-2,0)$.
\end{lemma}
The proof follows the same lines as the one in Lemma \ref{lem.Q4.bis}, with the difference that \eqref{interm13} is replaced now by
\begin{equation*}
\frac{p-m}{\sigma+2}y=w+o(w),
\end{equation*}
as it can be readily seen from the system \eqref{PSinf1.bu}, which leads to \eqref{beh.Q41.bu}. The reader is referred also to \cite[Lemma 4.2 and Lemma 4.3]{IMS23b} for details.

\section{Existence of blow-up profiles for $\sigma\in(-2,0)$}\label{sec.bu.negative}

This shorter section is devoted to the proof of Theorem \ref{th.blowup.super}, Part 1, which refers to the existence of blow-up self-similar solutions if $\sigma\in(-2,0)$ and for any $p\in(1,p_s(\sigma))$, noticing that $1>p_L(\sigma)$ if $\sigma\in(-2,0)$. We need one preparatory result before the main proof.
\begin{lemma}\label{lem.Z0.bu}
The orbit $r_0$ going out of $P_0$ inside the invariant plane $\{Z=0\}$ enters the critical point $Q_5$.
\end{lemma}
\begin{proof}
If we plug the constant $C=0$ into the general expansion \eqref{orbP0.bu} of the trajectories $r_C$, we observe that $r_0$ has $Y(\eta)\sim X(\eta)/N$ as $\eta\to\infty$, hence this orbit will go out into the positive half-plane $\{Y>0\}$. Moreover, the flow on the invariant plane $\{Z=0\}$ across the axis $\{Y=0\}$ has positive direction, thus the orbit $r_0$ will remain forever in the region $\{Y>0\}$. This readily gives from the first equation of \eqref{PSsyst.bu} that $\dot{X}(\eta)>0$ for any $\eta\in\real$, hence there exists $\lim\limits_{\eta\to\infty}X(\eta)$. The non-existence of another finite critical point of the system easily implies that the previous limit is infinite. Moreover, standard arguments (see for example Poincar\'e-Bendixon's theory, \cite[Theorem 5, Section 3.7]{Pe}) imply that the $\omega$-limit is a critical point at infinity, which can only be the attractor (for the invariant plane $\{Z=0\}$) $Q_5$, as claimed.
\end{proof}
We are now ready for the proof of the existence.
\begin{proof}[Proof of Theorem \ref{th.blowup.super}, Part 1]
It is sufficient to show that, for any $\sigma\in(-2,0)$ and any $p_s(\sigma)>p>1>p_L(\sigma)$, there exists at least a trajectory of the system \eqref{PSsyst.bu} connecting the critical points $P_0$ and $Q_1$. We know that in this range the orbit $r_{\infty}$ connects $P_0$ to $Q_3$ inside the invariant plane $\{X=0\}$ (see \cite[Section 5]{IMS23b}). We introduce the following sets:
\begin{equation*}
\begin{split}
&\mathcal{U}=\{C\in(0,\infty): {\rm on \ the \ orbit} \ r_C \ {\rm we \ have} \ Y(\eta)<0, \ {\rm for \ any} \ \eta\in\real \ {\rm and \ enters} \ Q_3\},\\
&\mathcal{W}=\{C\in(0,\infty): {\rm on \ the \ orbit} \ r_C \ {\rm there \ is \ some} \ \eta\in\real, \ {\rm such \ that} \ Y(\eta)>0\},\\
&\mathcal{V}=(0,\infty)\setminus(\mathcal{U}\cup\mathcal{W}).
\end{split}
\end{equation*}
Since $Q_3$ is a stable node, standard continuity arguments imply that $\mathcal{U}$ is nonempty and open (and in fact it contains a full interval $(C^*,\infty)$ for some $C^*>0$). It is then obvious that $\mathcal{W}$ is an open set by definition. Moreover, Lemma \ref{lem.Z0.bu} together with the continuous dependence on the parameter $C$ of the orbits $r_C$ on the stable manifold of $P_0$ gives that $\mathcal{W}\neq\emptyset$ (and more precisely, it contains an interval $(0,C_*)$ for some $C_*>0$). We conclude that $\mathcal{V}\neq\emptyset$. Let us pick now $C_0\in\mathcal{V}$. We show first that $Y(\eta)<0$ for any $\eta\in\real$ on the orbit $r_{C_0}$. Assume for contradiction that the latter is false. Since $C_0\notin\mathcal{W}$, it follows that $Y(\eta)\leq0$ for any $\eta\in\real$, and since $Y(\eta)<0$ for $\eta$ in a neighborhood of $-\infty$ (this follows from the local analysis near $P_0$ in Section \ref{sec.bu.general}), we infer that there exists a first tangency point $\eta_0\in\real$ such that $Y(\eta_0)=0$, $Y(\eta)<0$ if $\eta\in(-\infty,\eta_0)$ and $Y(\eta)<0$ for any $\eta\in(\eta_0,\eta_0+\epsilon)$ for some $\epsilon>0$. In other words, $Y(\eta_0)=0$ is a strict relative maximum for $Y(\eta)$ and thus $Y'(\eta_0)=0$ and $Y''(\eta_0)\leq0$. We deduce by evaluating the second equation of the system \eqref{PSsyst.bu} at $\eta=\eta_0$ that $X(\eta_0)=Z(\eta_0)$. Differentiating with respect to $\eta$ the same equation, evaluating at $\eta=\eta_0$ and taking into account that $Y(\eta_0)=Y'(\eta_0)=0$, we obtain that
$$
Y''(\eta_0)=X'(\eta_0)-Z'(\eta_0)=2X(\eta_0)-(\sigma+2)Z(\eta_0)=-\sigma X(\eta_0)>0,
$$
since $\sigma\in(-2,0)$. We have thus reached a contradiction, hence $Y(\eta)<0$ for any $\eta\in\real$ on the orbit $r_{C_0}$. Since $C_0\notin\mathcal{U}$, it follows that the orbit $r_{C_0}$ does not enter the node $Q_3$. We next show that $Y(\eta)$ is bounded from below on the orbit $r_{C_0}$. Indeed, the direction of the flow of the system \eqref{PSsyst.bu} across the plane $\{Y=-(N-2)/m\}$, respectively across the plane $\{Y=-(\sigma+2)/(p-m)\}$, is given by the sign of the expressions
\begin{equation}\label{flow.planes}
F_1(X,Z)=\frac{(N-2)(p_c(\sigma)-p)}{m(\sigma+2)}X-Z, \ F_2(X,Z)=-\frac{(\sigma+2)(N-2)(p_c(\sigma)-p)}{(p-m)^2}-Z,
\end{equation}
and we remark that either $F_2(X,Z)<0$ if $p\in(1,p_c(\sigma))$ or $F_1(X,Z)<0$ if $p\geq p_c(\sigma)$. Assume for contradiction that there exists $\eta_1\in\real$ such that either $p\geq p_c(\sigma)$ and $Y(\eta_1)<-(N-2)/m$, or $p\in(1,p_c(\sigma))$ and $Y(\eta_1)<-(\sigma+2)/(p-m)$ along the orbit $r_{C_0}$. It is easy to observe that in both cases, $Y'(\eta)<0$ if $Y(\eta)<-(N-2)/m$ for $p\geq p_c(\sigma)$ or if $Y(\eta)<-(\sigma+2)/(p-m)$ for $p\in(1,p_c(\sigma))$, hence $Y'(\eta)<0$ for any $\eta\in(\eta_1,\infty)$. Moreover, in both cases we also notice that $Z'(\eta)<0$ if $Y(\eta)<-(\sigma+2)/(p-m)$, which is also fulfilled if $p\geq p_c(\sigma)$ and $Y(\eta)<-(N-2)/m\leq-(\sigma+2)/(p-m)$, and we deduce that $Z'(\eta)<0$ for any $\eta\in(\eta_1,\infty)$. Arguments as in \cite[Proposition 4.10]{ILS23b} and its proof then imply that $r_{C_0}$ has a critical point as limit as $\eta\to\infty$, and we easily get that such point might only be $Q_3$, contradicting the fact that $C_0\in\mathcal{V}$. This contradiction allows us to conclude that either $p\in(1,p_c(\sigma))$ and $Y(\eta)\geq-(\sigma+2)/(p-m)$ or $p\geq p_c(\sigma)$ and $Y(\eta)\geq-(N-2)/m$, for any $\eta\in\real$. In any of the two cases, $Y(\eta)>-2/(1-m)$, for any $\eta\in\real$, which also gives $X'(\eta)>0$, thus $X(\eta)$ is increasing on the orbit $r_{C_0}$ and there exists $\lim\limits_{\eta\to\infty}X(\eta)$. If we again let
$$
g(\xi):=\xi^{(\sigma+2)/(p-m)}f(\xi)
$$
and compute the equation satisfied by $g$, we obtain a similar one to \eqref{interm11} with only one sign in front of a term involving $g'(\xi)$ changed, thus we conclude as in Part (b) of the proof of Theorem \ref{th.global.super} that $Z(\eta)$ is monotone on some interval $(\eta_1,\infty)$ and thus there exists $\lim\limits_{\eta\to\infty}Z(\eta)$. Arguments as in \cite[Proposition 4.10]{ILS23b} lead to the fact that the orbit $r_{C_0}$ must have either a single finite critical point as limit, or an $\omega$-limit belonging to the infinity of the space. Since there are no finite critical points in the strips limiting $Y(\eta)$ as above, the latter is true. Since $Y(\eta)$ is bounded from below and negative, we conclude on the one hand that $Y(\eta)/X(\eta)\to0$ as $\eta\to\infty$ on the orbit $r_{C_0}$, and on the other hand that
$$
\sigma+2+(p-m)Y(\eta)-2-(1-m)Y(\eta)=\sigma+(p-1)Y(\eta)<\sigma<0, \qquad \eta\in\real,
$$
which also gives, together with the first and third equation in the system \eqref{PSsyst.bu}, that
$$
\lim\limits_{\eta\to\infty}\frac{Z(\eta)}{X(\eta)}=0.
$$
The change of variable \eqref{change2} then leads to the critical point $Q_1$ as limit point. We thus conclude the proof.
\end{proof}

\section{Non-existence of blow-up profiles for $m\in[m_c,1)$ and $\sigma\geq0$}\label{sec.cons}

This section is devoted to the proof of Parts 2 and 3 in Theorem \ref{th.blowup.super}. The idea is to construct geometric barriers for the flow of the system \eqref{PSsyst.bu} in form of combinations of planes and surfaces. We split the proof into several steps for convenience.
\begin{proof}
Fix at first $\sigma>0$ and $p\in(p_F(\sigma),p_s(\sigma))$. The idea is to show that no orbit $r_C$ defined in \eqref{orbP0.bu} connects to the critical point $Q_1$. Let us first observe from \eqref{beh.Q1f} together with the definition of $Y$ in \eqref{PSchange} that approaching the point $Q_1$ translates into $Y(\eta)\to-(\sigma+2)/(p-m)$ as $\eta\to\infty$ in the system \eqref{PSsyst.bu}. We can thus say by ``abuse of language" that $Q_1$ ``belongs" to the half-space $\{Y<0\}$.

\medskip

\noindent \textbf{Step 1.} Consider the plane of equation
\begin{equation}\label{plane1}
Z={\rm pln}(X,Y):=(N+\sigma)\left(\frac{X}{N}-Y\right), \qquad Y>0,
\end{equation}
which is tangent to the two-dimensional unstable manifold of $P_0$. Taking the normal direction as
$$
\overline{n}=\frac{1}{N}(-(N+\sigma),N(N+\sigma),N),
$$
the direction of the flow of the system \eqref{PSsyst.bu} across the plane \eqref{plane1} is given by the sign of the scalar product between the vector field of the system and the normal direction $\overline{n}$
\begin{equation}\label{interm14}
F_1(X,Y)=-Np(N+\sigma)Y\left[Y-\frac{(p-m)(N+\sigma)+L}{Np(\sigma+2)}X\right].
\end{equation}
Since $Z\geq0$, we have $Y\leq X/N$ on the plane \eqref{plane1}, hence
$$
Y-\frac{(p-m)(N+\sigma)+L}{Np(\sigma+2)}X\leq\frac{X}{N}\left[1-\frac{(p-m)(N+\sigma)+L}{p(\sigma+2)}\right]=\frac{p_F(\sigma)-p}{p(\sigma+2)}X\leq0.
$$
We conclude from \eqref{interm14} that $F_1(X,Y)\geq0$ on the plane \eqref{plane1} if $Y>0$.

\medskip

\noindent \textbf{Step 2.} We prove now that the orbits $r_C$, $C\in(0,\infty)$, of the unstable manifold of $P_0$ go out into the region $\mathcal{D}=\{Z>{\rm pln}(X,Y)\}$. To this end, we need to obtain the second order in the Taylor expansion near the origin of this unstable manifold. We set
\begin{equation}\label{unstable}
Z=(N+\sigma)\left(\frac{X}{N}-Y\right)+aX^2+bXY+cY^2+o(|(X,Y)|^2),
\end{equation}
and by following \cite[Section 2.7]{Shilnikov}, we find the following coefficients
\begin{equation}\label{interm15}
\begin{split}
&a=-\frac{\sigma(N+\sigma)A(m,N,p,\sigma)}{N^2(\sigma+2)(N+2)(N+\sigma+2)(N+2\sigma+2)}, \\
&b=-\frac{(N+\sigma)A(m,N,p,\sigma)}{N(\sigma+2)(N+\sigma+2)(N+2\sigma+2)}, \qquad c=-\frac{(N+\sigma)p}{N+2\sigma+2},
\end{split}
\end{equation}
where
$$
A(m,N,p,\sigma)=-\frac{N(N^2+3N\sigma+4N+2\sigma+4)(p-p_F(\sigma))+(\sigma+2)(N+2)(N+2\sigma+2)}{N}.
$$
We want to show that $aX^2+bXY+cY^2>0$ in a sufficiently small neighborhood of $(0,0,0)$ and very close to the plane $Y=X/N$, which is the first approximation order of the orbits $r_C$ defined in \eqref{orbP0.bu} for $\sigma>0$. At a formal level, we notice that, with the coefficients defined in \eqref{interm15}, we get
$$
aX^2+bXY+cY^2\Big|_{Y=X/N}=\frac{(N+\sigma)(p-p_F(\sigma))}{N(N+2)(\sigma+2)}X^2>0,
$$
since $p>p_F(\sigma)$. To make this rigorous, we subtract these two terms and find
\begin{equation*}
\begin{split}
aX^2+bXY+cY^2&-\frac{(N+\sigma)(p-p_F(\sigma))}{N(N+2)(\sigma+2)}X^2=-\frac{(N+\sigma)(NY-X)}{N^2(\sigma+2)(N+\sigma+2)(N+2\sigma+2)}\\
&\times\left[B(m,N,p,\sigma)X+N(\sigma+2)(N+\sigma+2)Y\right]=o(|(X,Y)|^2),
\end{split}
\end{equation*}
since $NY-X=o(|(X,Y)|)$ as $(X,Y)\to(0,0)$ on the unstable manifold of $P_0$, where in the previous expression,
\begin{equation*}
B(m,N,p,\sigma)=N(N+2\sigma+2)(m-p)+(p+2)\sigma(\sigma+2)+(N+2)(\sigma+2).
\end{equation*}
It thus follows that indeed, for $p>p_F(\sigma)$, in a sufficiently small neighborhood of $P_0$, its unstable manifold \eqref{unstable} goes out into the region $\mathcal{D}$ and, by the outcome of Step 1, it will remain there at least until crossing the plane $\{Y=0\}$. In particular, we get that any trajectory $r_C$ with $C\in(0,\infty)$ either remains forever in the half-space $\{Y>0\}$ (in which case it might not connect to $Q_1$) or it crosses the plane $\{Y=0\}$ at a point such that
\begin{equation}\label{cross}
Z>\frac{N+\sigma}{N}X.
\end{equation}

\medskip

\noindent \textbf{Step 3.} Consider the surface
\begin{equation}\label{surface}
Z={\rm sup}(X,Y):=(N+\sigma)\left(\frac{X}{N}-Y\right)-pY^2+\frac{(p-m)(N+\sigma)}{N(\sigma+2)}XY,
\end{equation}
for $Y\leq0$. We infer from \eqref{cross} that all the orbits on the unstable manifold of $P_0$ crossing the plane $\{Y=0\}$ enter the region $\{Z>{\rm sup}(X,Y)\}$. Taking as normal vector to the surface \eqref{surface}
$$
\overline{N}=\left(-\frac{(N+\sigma)(\sigma+2+(p-m)Y)}{N(\sigma+2)},-\frac{(p-m)(N+\sigma)}{N(\sigma+2)}X+2pY+N+\sigma,1\right),
$$
we obtain that the direction of the flow of the system \eqref{PSsyst.bu} across the surface \eqref{surface} is given by the sign of the expression
\begin{equation}\label{interm16}
\begin{split}
F_2(X,Y)&=\left(Y+\frac{\sigma+2}{p-m}\right)\left[\frac{\sigma(p-m)^2(N+\sigma)}{N^2(\sigma+2)^2}X^2+\frac{(p-m)C(m,N,p,\sigma)}{N(\sigma+2)}XY\right.\\
&\left.+(p-m)pY^2\right],
\end{split}
\end{equation}
with
$$
C(m,N,p,\sigma)=-(N+\sigma)(1-m)-2p\sigma<0.
$$
We readily infer from \eqref{interm16} that $F_2(X,Y)>0$ in the region $\{-(\sigma+2)/(p-m)<Y<0\}$, hence, in order to be allowed to cross the surface \eqref{surface}, an orbit on the unstable manifold of $P_0$ must first reach the plane $\{Y=-(\sigma+2)/(p-m)\}$. Let us also remark here that the intersection between the latter plane and the surface \eqref{surface} is given by the line $\{Z=Z(P_2),Y=-(\sigma+2)/(p-m)\}$, which is positive only if $p>p_c(\sigma)$ and is in this case the parallel line to the $X$-axis through the critical point $P_2$. For the easiness of the reading, we give in Figure \ref{fig3} a picture of both the plane \eqref{plane1} and the surface \eqref{surface}.

\begin{figure}[ht!]
  \begin{center}
  \includegraphics[width=11cm,height=7.5cm]{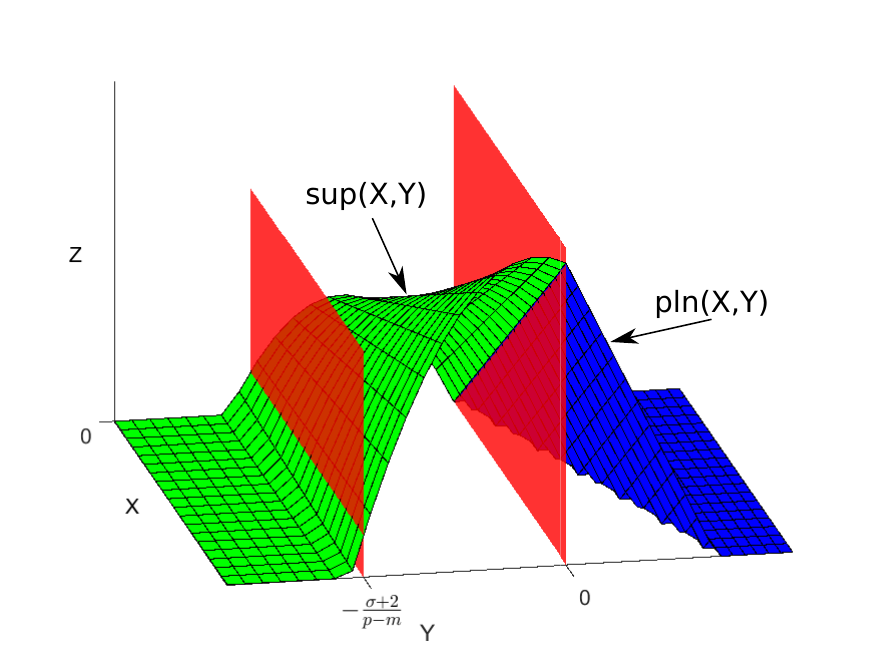}
  \end{center}
  \caption{The plane defined in \eqref{plane1} and the surface defined in \eqref{surface}.}\label{fig3}
\end{figure}

\medskip

\noindent \textbf{Step 4.} We prove in this step that, if either $1<p\leq p_c(\sigma)$ and $Z>0$ or $p>p_c(\sigma)$ and $Z>Z(P_2)$ (where we recall that $Z(P_2)$ is defined in \eqref{zp2}), the orbits entering $Q_1$ and overpassing such values of $Z(\eta)$ do that through the region $\{Z<{\rm sup}(X,Y)\}$. To this end, we first write the surface \eqref{surface} in the $(x,y,z)$ variables of the system \eqref{PSinf1.bu} by recalling the change of variable \eqref{change2}, obtaining the equation
\begin{equation}\label{interm17}
\frac{N+\sigma}{N}x-(N+\sigma)xy+\frac{(p-m)(N+\sigma)}{N(\sigma+2)}y-py^2-xz=0.
\end{equation}
We want to write the surface in the form $y=h(x,z)$ for a suitable function $h$. To this end, we solve the equation \eqref{interm17} in terms of $y$ and compute its Taylor expansion up to second order, getting
\begin{equation}\label{surface2}
y=-\frac{\sigma+2}{p-m}x+\frac{N(\sigma+2)^2(N-2)(p_c(\sigma)-p)}{(N+\sigma)(p-m)^3}x^2+\frac{(\sigma+2)N}{(N+\sigma)(p-m)}xz+o(|(x,z)|^2).
\end{equation}
We want thus to compare, in a neighborhood of the critical point $Q_1$ (that is, of the origin in the system \eqref{PSinf1.bu}) the expansion of the surface $\eqref{surface2}$ with the expansion (up to the same order) of the center manifold \eqref{center.Q1} of the point. By relabelling $y(Q_1)$ the variable $y$ corresponding to the center manifold in \eqref{center.Q1} and $y({\rm sup})$ the variable $y$ corresponding to the surface \eqref{surface2}, we find
\begin{equation}\label{interm18}
\begin{split}
y(Q_1)-y({\rm sup})&=\frac{\sigma x}{N+\sigma}\left[\frac{(\sigma+2)^2(N-2)(p_c(\sigma)-p)}{(p-m)^3}x+\frac{\sigma+2}{p-m}z\right]+o(|(x,z)|^2)\\
&=\frac{\sigma(\sigma+2)x^2}{(N+\sigma)(p-m)}(Z-Z(P_2))+o(|(x,z)^2|)>0
\end{split}
\end{equation}
where we have used in the last line that $Z=z/x$ and noticed that $Z(P_2)<0$ if it is defined (in the same way as in \eqref{zp2}) also for $1<p<p_c(\sigma)$. We conclude from \eqref{interm18} and from the equation of the center manifold \eqref{center.Q1} together with the change of variable \eqref{change2} that the trajectories entering $Q_1$ on the center manifold come from the region $\{Z<{\rm sup}(X,Y)\}$ and with $Y>-(\sigma+2)/(p-m)$.

\medskip

\noindent \textbf{Step 5. Conclusion for $\sigma>0$ and $p\in(p_F(\sigma),p_s(\sigma))$.} From the previous steps, we deduce that for any $p\in(p_F(\sigma),p_s(\sigma))$, the orbits on the unstable manifold of $P_0$ cannot reach the critical point $Q_1$ unless they do it either with $Z<Z(P_2)$ or after crossing first the plane $\{Y=-(\sigma+2)/(p-m)\}$. If $p_F(\sigma)<p\leq p_c(\sigma)$, the former is not possible as $Z(P_2)\leq0$, while the direction of the flow of the system \eqref{PSsyst.bu} across the plane $\{Y=-(\sigma+2)/(p-m)\}$ which has always negative sign according to the expression $F_2(X,Z)$ in \eqref{flow.planes} ensures that the latter is also impossible: once crossed the plane $\{Y=-(\sigma+2)/(p-m)\}$, no trajectory can come back to the half-space $\{Y>-(\sigma+2)/(p-m)\}$ in order to reach the point $Q_1$ afterwards. If $p_c(\sigma)<p<p_s(\sigma)$, it has been proved in \cite[Section 6]{IMS23b} that all the trajectories on the unstable manifold of $P_0$ stay forever in the exterior of the cylinder \eqref{cylinder} (as one can readily verify from \eqref{flow.cyl.ext}). Since in a neighborhood of $Y=-(\sigma+2)/(p-m)$ the height of the cylinder \eqref{cylinder} is larger than $Z(P_2)$, we conclude that the orbits on the unstable manifold of $P_0$ will reach $Z(\eta)>Z(P_2)$ at some finite $\eta\in\real$, since $\dot{Z}(\eta)>0$ while $Y(\eta)>-(\sigma+2)/(p-m)$. Thus, the only way to reach the critical point $Q_1$ remains crossing first the plane $\{Y=-(\sigma+2)/(p-m)\}$, diminishing the value of $Z(\eta)$ and then re-entering the half-space $\{Y>-(\sigma+2)/(p-m)\}$. But the latter is again impossible, as a simple inspection of the expression $F_2(X,Z)$ in \eqref{flow.planes} shows that it is always negative provided $Z>Z(P_2)$, which is always fulfilled at $Y=-(\sigma+2)/(p-m)$ and in the exterior region to the cylinder \eqref{cylinder}. We have thus proved that, if $\sigma>0$ and $p_F(\sigma)<p<p_s(\sigma)$, no trajectory of the system \eqref{PSsyst.bu} can connect $P_0$ to $Q_1$.

\medskip

\noindent \textbf{Step 6. Proof for $\sigma=0$ and $p\in(1,p_s(\sigma))$.} Fix now $\sigma=0$ and notice that, in this case, $p_L(\sigma)=1$. Moreover, we observe that, on the one hand, the intersection of the surface \eqref{surface} with the plane $\{Y=0\}$ reduces in this case to the line $\{X=Z\}$, and on the other hand, the direction of the flow of the system \eqref{PSsyst.bu} across the plane $\{Y=0\}$ is given by the sign of $X-Z$. We thus infer that any orbit (stemming from $P_0$ or any other critical point) which crosses the plane $\{Y=0\}$ from the positive into the negative half-space can only do this crossing at points with $Z>X$, hence entering directly the region $\{Z>{\rm sup}(X,Y)\}$ in the half-space $\{Y<0\}$. This fact allows us to remove the technical restriction $p>p_F(\sigma)$ needed in the first steps of the proof for $\sigma>0$. The same holds true for the trajectories going out of $P_0$ directly into the half-space $\{Y<0\}$, corresponding to the orbits $(r_C)$ with $C>1$ in \eqref{orbP0.bu}, since
$$
Z(\eta)\sim CX(\eta)>{\rm sup}(X(\eta),Y(\eta)), \qquad {\rm as} \ \eta\to-\infty.
$$

One more technical problem appears, as we can see from \eqref{interm18} that the second order Taylor approximation of the surface \eqref{surface2} and of the center manifold \eqref{center.Q1} of $Q_1$ match perfectly for $\sigma=0$. We thus go one step further, to the third order Taylor approximation. On the one hand, the third order approximation of the surface obtained by expressing $y$ in terms of $x$ and $z$ in \eqref{interm17} with $\sigma=0$ is given by
\begin{equation*}
\begin{split}
y({\rm sup})&=-\frac{2}{p-m}x+\frac{4(mN-Np+2p)}{(p-m)^3}x^2+\frac{2}{p-m}xz\\
&-\frac{8(mN-Np+2p)(mN-Np+4p)}{(p-m)^5}x^3\\&-\frac{4(mN-Np+4p)}{(p-m)^3}x^2z+o(|(x,z)|^3).
\end{split}
\end{equation*}
On the other hand, we can compute once more, following \cite[Theorem 3, Section 2.5]{Carr}, this time the third order approximation of the equation of the center manifold of $Q_1$. To this end, what one does in practice is to plug in the ansatz
\begin{equation*}
\begin{split}
y(Q_1)&=-\frac{2}{p-m}x+\frac{4(mN-Np+2p)}{(p-m)^3}x^2+\frac{2}{p-m}xz\\
&+ax^3+bx^2z+cxz^2+dz^3+o(|(x,z)|^3),
\end{split}
\end{equation*}
and employ the equation of the center manifold to identify the similar coefficients, which in fact is completely equivalent to require that the flow of the system \eqref{PSinf1.bu} on the manifold given by the previous third order equation only has terms of higher order than three. By performing some straightforward (but a bit tedious, for which we employed a symbolic calculation program) calculations, one gets in the end the Taylor approximation
\begin{equation*}
\begin{split}
y(Q_1)&=-\frac{2}{p-m}x+\frac{4(mN-Np+2p)}{(p-m)^3}x^2+\frac{2}{p-m}xz\\
&-\frac{8(mN-Np+2p)(mN-Np+4p+2(m-1))}{(p-m)^5}x^3\\&-\frac{4(mN-Np+4p+2(m-1))}{(p-m)^3}x^2z+o(|(x,z)|^3),
\end{split}
\end{equation*}
and thus, by subtracting the two approximations we obtain
\begin{equation}\label{interm19}
y(Q_1)-y({\rm sup})=\frac{8(1-m)x^2}{(p-m)^5}\left[2(N-2)(p_c(0)-p)x+(p-m)^2z\right],
\end{equation}
which again is obviously positive if $1<p\leq p_c(0)$ and is positive provided $Z=z/x>Z(P_2)$ if $p_c(0)<p<p_s(0)$. Since the argument with the cylinder still holds true, we are exactly in the same position as in Step 4 above after concluding the positivity of the expression in \eqref{interm18}. We thus end up the proof for $\sigma=0$ in the same way as in Step 5, by deducing from \eqref{interm19}, the limiting cylinder \eqref{cylinder} and the flow of the system \eqref{PSsyst.bu} across the plane $\{Y=-(\sigma+2)/(p-m)\}$, that no trajectory may exist connecting the critical points $P_0$ and $Q_1$.

\medskip

\noindent \textbf{Step 7. Proof for $\sigma>0$ and $p\in(p_L(\sigma),p_*(\sigma))$.} We have left this step at the end, since the geometric construction will differ with respect to the previous one. Let us start from the system \eqref{PSinf1.bu} in variables $(x,y,z)$ and consider the following surface
\begin{equation}\label{surface4}
\frac{(\sigma+2)(p-1)}{L}x+\frac{(p-m)(p-1)}{L}y-xz=0,
\end{equation}
where we recall that $L$ is defined in \eqref{const.L}. Taking as normal direction to this surface the vector
$$
\overline{n}=\left(\frac{(\sigma+2)(p-1)}{L}-z,\frac{(p-m)(p-1)}{L},-x\right),
$$
we prove next that the direction of the flow of the system \eqref{PSinf1.bu} across this surface is in the opposite direction to the normal. Indeed, the direction of the flow is given by the sign of the scalar product between the vector field of the system \eqref{PSinf1.bu} and the vector $\overline{n}$ given above, which gives the expression
\begin{equation*}
\begin{split}
F(x,y)&=-\frac{(p-1)(\sigma+2)\sigma}{L}x^2-\frac{(p-1)(p-m)(m+p-1)}{L}y^2\\
&+\frac{(p-1)(Nm-Np+m\sigma-2p\sigma-2m+\sigma+2)}{L}xy\\
&-\frac{(p-1)(p-m)^2\sigma}{L^2}x-\frac{(p-1)(p-m)^3\sigma}{L^2(\sigma+2)}y.
\end{split}
\end{equation*}
On the one hand, it is obvious that the sum of the two linear terms in the expression of $F(x,y)$ is non-positive if $y\geq-(\sigma+2)x/(p-m)$. On the other hand, we handle the three quadratic terms in the expression of $F(x,y)$ in the following way: since we are only interested in the region $\{y\leq0\}$, we take $x^2$ as common factor between them and set $c=-y/x$, obtaining thus that the sign of the quadratic part of $F(x,y)$ is given by the sign of the second degree polynomial
\begin{equation}\label{polc}
\begin{split}
P(c)&:=-\frac{p-1}{L}\left[(m+p-1)(p-m)c^2+(Nm-Np+m\sigma-2p\sigma-2m+\sigma+2)c\right.\\&\left.+\sigma(\sigma+2)\right],
\end{split}
\end{equation}
and we obtain after rather straightforward manipulations that $P(c)\leq0$ for any $c$ such that
$$
-\frac{\sigma+2}{p-m}\leq c\leq-\frac{\sigma}{p-1},
$$
since the quadratic factor in brackets in \eqref{polc} is nonnegative in the above mentioned range of $c$. This is the part where the technical limitation appears. Indeed, we readily observe that
$$
P\left(-\frac{\sigma}{p-1}\right)=-\frac{(p-1)(\sigma+2)(m(N+\sigma)-p(N-2))}{L(\sigma+2)}<0,
$$
since $p<p_F(\sigma)<p_c(\sigma)$, but in change
$$
P\left(-\frac{\sigma+2}{p-m}\right)=-\frac{\sigma(p-m)((N-2)(p-1)-m\sigma)}{L(p-1)},
$$
and this is negative if and only if
$$
p_L(\sigma)<p<p_*(\sigma)=1+\frac{m\sigma}{N-2}.
$$
This is why we need to impose what we believe to be only a technical restriction in this step, that $p<p_*(\sigma)$. The rest of the proof of the non-positivity of $P(c)$ is based on comparing the point of maximum of the parabola \eqref{polc} with $c=-\sigma/(p-1)$ and with $c=-(\sigma+2)/(p-m)$ to show that it cannot lie between these two values, and it is here where the restriction $N\geq4$ is needed. We omit here these technical calculations.

All this analysis proves that $F(x,y)\leq0$ provided
\begin{equation}\label{interm32}
-\frac{\sigma+2}{p-m}x\leq y\leq -\frac{\sigma}{p-1}x, \qquad x>0.
\end{equation}
This, together with the decreasing direction of the normal vector $\overline{n}$ (observe that the $z$ component is always negative), gives that no trajectory of the system \eqref{PSinf1.bu} may cross the surface \eqref{surface4} with a decreasing value of $z$. Observe then that the surface \eqref{surface4} can be also written as
\begin{equation}\label{surface4.bis}
y=y_{\rm sup}=-\frac{\sigma+2}{p-m}x+\frac{L}{(p-1)(p-m)}xz,
\end{equation}
which, comparing to the second order approximation of the center manifold of $Q_1$ given in \eqref{center.Q1}, gives
$$
y_{\rm sup}-y_{Q_1}=-\frac{\sigma}{p-1}xz-\frac{(\sigma+2)^2(N-2)(p_c(\sigma)-p)}{(p-m)^3}x^2<0,
$$
which, together with the direct proportionality of $y$ with respect to $z$ in \eqref{surface4.bis}, leads to the fact that the center manifold of the critical point $Q_1$ lies ``below" the surface \eqref{surface4} (in the sense of smaller values of $z$ for the same $x$ and $y$) in a neighborhood of $Q_1$. The direction of the flow of the system \eqref{PSinf1.bu} across the surface \eqref{surface4} implies that the same remains true at least in the range \eqref{interm32}. Noticing that the intersection of the plane $y=-\sigma x/(p-1)$ with the surface \eqref{surface4} is given exactly by the line $\{z=1\}$, we infer that any trajectory on the center manifold \eqref{center.Q1} of $Q_1$ must either lie forever in the region $\{y<-\sigma x/(p-1)\}$ or cross the plane $\{y=-\sigma x/(p-1)\}$ at a point with $z<1$.

Assume now for contradiction that there exists an orbit between $P_0$ and $Q_1$ in the system \eqref{PSinf1.bu}, with $\sigma>0$, $N\geq4$ and $p_L(\sigma)<p<p_*(\sigma)$. Such an orbit must first reach the half-space $\{y>0\}$ and then cross the plane $\{y=0\}$, which can be crossed from the positive to the negative side only at points with $z>1$, as seen from the second equation in \eqref{PSinf1.bu}. Afterwards, the value of $z$ still increases along the trajectory exactly until intersecting the plane $\{y=-\sigma x/(p-1)\}$, as it can be seen from the sign of $\dot{z}$ in the third equation of \eqref{PSinf1.bu}. Thus, such a crossing point must have coordinate $z>1$, which contradicts the fact that all trajectories entering $Q_1$ (and which do not lie forever in the region $\{y<-\sigma x/(p-1)\}$) cross this plane at points with $z<1$, as established above. This contradiction proves that there is no trajectory between $P_0$ and $Q_1$ and thus no self-similar blow-up solution in the specified range of $p$.
\end{proof}

\noindent \textbf{Remark.} Notice that $p_*(\sigma)$ is another critical exponent of the equation Eq. \eqref{eq1}, which comes from equating the exponents $-(N-2)/m$ and $-\sigma/(p-1)$, both of them giving some possible local behavior of profiles. In particular, when $p=p_*(\sigma)$, we have an explicit self-similar solution with vertical asymptote given by
$$
u(x,t)=\left(\frac{N-2}{m\sigma}\right)^{(N-2)/(m\sigma)}(T-t)^{-(N-2)/(m\sigma)}|x|^{-(N-2)/m}.
$$
The plane \eqref{plane1} and surface \eqref{surface} have been also employed with success for $m>1$ in the recent work \cite{IS25b}.

\section{Dimensions $N=1$ and $N=2$}\label{sec.N12}

Since all the previous analysis has been performed in dimension $N\geq3$, we are left with the lower dimensions $N=1$ and $N=2$. We notice first that, as we explained in the Introduction, we are always by default in the supercritical range $m>0\geq m_c$, while $p_c(\sigma)=p_s(\sigma)=+\infty$, thus we will always be as in the cases $p<p_c(\sigma)$ and $p<p_s(\sigma)$ in the previous analysis. There are some technical differences with respect to the local analysis of the systems \eqref{PSsyst} and \eqref{PSsyst.bu} in a neighborhood of the critical points $P_0$ and $P_1$, that we make precise below. Let us choose for the proofs, by convention, the system \eqref{PSsyst}, as the analysis for the system \eqref{PSsyst.bu} will be completely similar.
\begin{lemma}\label{lem.N2}
The critical points $P_0$ and $P_1$ coincide in dimension $N=2$. Denoting by $P_0$ the resulting point, it is a saddle-node presenting a leading three-dimensional center-unstable manifold tangent to the $Y$-axis and a non-leading two-dimensional unstable manifold. The orbits in the center-unstable manifold tangent to the $Y$-axis contain profiles with a vertical asymptote at $\xi=0$ of the form
\begin{equation}\label{beh.P0.N2}
f(\xi)=D(-\ln\,\xi)^{1/m}, \qquad {\rm as} \ \xi\to0, \qquad D>0,
\end{equation}
while the orbits in the unstable manifold contain profiles with one of the local behaviors \eqref{beh.P0f} or \eqref{beh.P0f} according to the sign of $\sigma$.
\end{lemma}
\begin{proof}
The linearization of the system \eqref{PSsyst} in a neighborhood of $P_0$ in dimension $N=2$ has the matrix
$$
M(P_0)=\left(
         \begin{array}{ccc}
           2 & 0 & 0 \\
           -1 & 0 & -1 \\
           0 & 0 & \sigma+2 \\
         \end{array}
       \right),
$$
having a two-dimensional unstable manifold and one-dimensional center manifolds. Using standard theory of center manifolds (see for example \cite[Section 3.2]{GH}), we readily infer that all the center manifolds are tangent to the center eigenspace, which is the invariant $Y$-axis (intersection of the invariant planes $\{X=0\}$ and $\{Z=0\}$), and the direction of the flow on every center manifold is given by the dominating term in the second equation of the system \eqref{PSsyst} near $P_0$, that is
$$
\dot{Y}=-mY^2<0,
$$
thus it is stable in the half-space $\{Y>0\}$ and unstable in the half-space $\{Y<0\}$. We thus obtain a saddle-node (see for example \cite[Section 3.2 and Section 3.4]{GH}) with the saddle sector in $\{Y>0\}$ and the unstable node sector in $\{Y<0\}$. The orbits contained in the node sector satisfy the conditions
\begin{equation}\label{interm26}
\lim\limits_{\eta\to-\infty}\frac{X(\eta)}{Y(\eta)}=\lim\limits_{\eta\to-\infty}\frac{Z(\eta)}{Y(\eta)}=0.
\end{equation}
We want to deduce from \eqref{interm26} the local behavior \eqref{beh.P0.N2}. Taking into account that $\eta=\ln\,\xi$, we have the following limits over trajectories that satisfy \eqref{interm26}:
\begin{equation}\label{lim1}
\lim\limits_{\xi\to0}\frac{(f^m)'(\xi)}{\alpha\xi f(\xi)}=\lim\limits_{\xi\to0}\frac{mf^{m-1}(\xi)\xi f'(\xi)}{\alpha\xi^2 f(\xi)}=\lim\limits_{\xi\to0}\frac{Y(\xi)}{X(\xi)}=\lim\limits_{\eta\to-\infty}\frac{Y(\eta)}{X(\eta)}=-\infty,
\end{equation}
then
\begin{equation}\label{lim2}
\lim\limits_{\xi\to0}\frac{(f^m)'(\xi)}{\xi^{\sigma+1}f^p(\xi)}=\lim\limits_{\xi\to0}\frac{mf^{m-1}(\xi)\xi f'(\xi)}{\xi^{\sigma+2}f^p(\xi)}=\lim\limits_{\xi\to0}\frac{Y(\xi)}{Z(\xi)}=\lim\limits_{\eta\to-\infty}\frac{Y(\eta)}{Z(\eta)}=-\infty,
\end{equation}
and finally
\begin{equation}\label{lim3}
\lim\limits_{\xi\to0}\frac{(f^m)'(\xi)}{\beta\xi^2f'(\xi)}=\lim\limits_{\xi\to0}\frac{\alpha}{\beta X(\xi)}=\lim\limits_{\eta\to-\infty}\frac{\alpha}{\beta X(\eta)}=+\infty.
\end{equation}
We infer from \eqref{lim1}, \eqref{lim2} and \eqref{lim3} that the last three terms in the differential equation \eqref{ODE.forward} are in a first order of approximation negligible with respect to the second term. Joining the first and the second terms and integrating the resulting differential equation in a neighborhood of $\xi=0$ gives the local behavior \eqref{beh.P0.N2}. Finally, the orbits belonging to the unstable manifold generated by the first and third eigenvalues of $M(P_0)$ are completely similar as the ones in Lemma \ref{lem.P0} and the analysis performed there gives the behaviors \eqref{beh.P0f} or \eqref{beh.P0f} according to the sign of $\sigma$. It is obvious that the same happens if we consider dimension $N=2$ in the system \eqref{PSsyst.bu}.
\end{proof}
We analyze now the critical points $P_0$ and $P_1$ in dimension $N=1$, where we have the further restriction $\sigma>-1$ (instead of $\sigma>-2$ as usual).
\begin{lemma}\label{lem.N1}
In dimension $N=1$, the critical point $P_0$ is a stable node, while the critical point $P_1$ is a saddle point having a two-dimensional unstable manifold and a one-dimensional stable manifold contained in the $Y$-axis. The profiles contained in the orbits stemming from $P_0$ have either the local behavior
\begin{equation}\label{beh.P0.N1}
f(\xi)\sim\left[A-K\xi\right]^{-2/(1-m)}, \qquad {\rm as} \ \xi\to0,
\end{equation}
with $A>0$ and $K\in\real\setminus\{0\}$ arbitrary constants, or one of the local behaviors \eqref{beh.P0f} or \eqref{beh.P0f} according to the sign of $\sigma$. The profiles contained in the unstable manifold of the critical point $P_1$ have the local behavior
\begin{equation}\label{beh.P1.N1}
f(\xi)\sim K\xi^{1/m}, \qquad {\rm as} \ \xi\to0, \qquad K>0.
\end{equation}
\end{lemma}
\begin{proof}
We already know that in a neighborhood of $P_0$ we have $Z\sim CX^{(\sigma+2)/2}$, thus, the first order of approximation of the orbits going out of $P_0$ is given by the equation of the dominating terms
$$
\frac{dY}{dX}=\frac{Y+X-CX^{(\sigma+2)/2}}{2X},
$$
which gives by integration that
\begin{equation}\label{interm27}
Y(\eta)\sim KX(\eta)^{1/2}+X(\eta)-\frac{C}{\sigma+1}X(\eta)^{(\sigma+2)/2}+l.o.t., \qquad {\rm as} \ \eta\to-\infty.
\end{equation}
When the integration constant $K\neq0$, since $\sigma+2>1$, we observe that the first term in the right hand side of \eqref{interm27} dominates, and by passing to profiles and integrating over $(0,\xi)$ we readily obtain the local behavior \eqref{beh.P0.N1}. When the integration constant $K=0$, we are in the same situation as in Lemma \ref{lem.P0} and we can proceed exactly as there to find the local behaviors \eqref{beh.P0f} or \eqref{beh.P0f} depending on the sign of $\sigma$. Concerning the critical point $P_1$, the linearization of the system \eqref{PSsyst} in a neighborhood of it becomes
$$
M(P_1)=\left(
         \begin{array}{ccc}
           \frac{m+1}{m} & 0 & 0 \\[1mm]
           -\frac{m(\sigma+1)+p}{m(\sigma+2)} & -1 & -1 \\[1mm]
           0 & 0 & \frac{m(\sigma+1)+p}{m} \\
         \end{array}
       \right),
$$
with two positive eigenvalues and one negative eigenvalue. The stable manifold is unique and tangent to the negative eigenspace, hence it is contained in the invariant $Y$-axis, while the unstable manifold is two-dimensional and contains orbits such that $Y(\eta)\to 1/m$ as $\eta\to-\infty$. It is then easy to get the local behavior \eqref{beh.P1.N1} by undoing the change \eqref{PSchange}.
\end{proof}
We remark that the same as above is also true when considering the orbits of the system \eqref{PSsyst.bu}, since the leading terms driving us to the new local behaviors \eqref{beh.P0.N2}, \eqref{beh.P0.N1} and \eqref{beh.P1.N1} do not depend on the terms $X$ and $(p-m)/(\sigma+2)XY$ that change sign in the equation for $\dot{Y}$ from \eqref{PSsyst} to \eqref{PSsyst.bu}. We also notice from the proofs above that the orbits we are interested in, giving the interesting local behaviors as $\xi\to0$, are tangent to the eigenspace corresponding to the eigenvalues $\lambda_1$ and $\lambda_3$ of the matrix $M(P_0)$, and its local behavior in a neighborhood of $P_0$ is still given by \eqref{orbP0.global}, respectively \eqref{orbP0.bu}. By replacing the usual expression $(N-2)(p_c(\sigma)-p)$ by $m(N+\sigma)-p(N-2)>0$ in dimensions $N=1$ and $N=2$, a simple inspection of the proofs of Theorems \ref{th.global.super} and \ref{th.blowup.super} in the range $p_L(\sigma)<p<\infty=p_s(\sigma)$ allows us to conclude that they can be repeated identically and thus also hold true in dimensions $N=1$ and $N=2$.

\bigskip

\noindent \textbf{Acknowledgements} R. G. I. and A. S. are partially supported by the Project PID2020-115273GB-I00 and by the Grant RED2022-134301-T funded by MCIN/AEI/10.13039/ \\ 501100011033 (Spain).

\bibliographystyle{plain}

\end{document}